\title{A hypergraph Tur\'an theorem via lagrangians of intersecting families}
\author{{Dan Hefetz \thanks{School of Mathematics, University of Birmingham, Edgbaston, Birmingham, B15 2TT, UK. Email: d.hefetz@bham.ac.uk.}} \quad {Peter Keevash \thanks{School of Mathematical Sciences, Queen Mary University of London, Mile End Road, London E1 4NS, England. Email: p.keevash@qmul.ac.uk. Research supported in part by ERC grant 239696 and EPSRC grant EP/G056730/1.}}}

\documentclass[11pt]{article}
\usepackage{amsmath,amssymb,latexsym,color,epsfig,a4,enumerate,graphicx}

\newif\ifnotesw\noteswtrue

\newcommand{\mc}[1]{\mathcal{#1}}

\newcommand{\brac}[1]{\left( #1 \right)}

\newcommand{\bfl}[1]{\left\lfloor #1 \right\rfloor}

\newcommand{\sub}{\subseteq}
\newcommand{\subn}{\subsetneq}

\newcommand{\ex}{\mbox{ex}}

\newcommand{\sm}{\setminus}

\newcommand{\eps}{\varepsilon}
\newcommand{\es}{\emptyset}
\newcommand{\pl}{\partial}

\newcommand{\aA}{\alpha}
\newcommand{\bB}{\beta}
\newcommand{\gG}{\gamma}
\newcommand{\dD}{\delta}

\newcommand{\lL}{\lambda}

\newcommand{\Ss}{\Sigma}

\newtheorem{theorem}{Theorem}[section]
\newtheorem{lemma}[theorem]{Lemma}

\newtheorem{observation}[theorem]{Observation}
\newtheorem{corollary}[theorem]{Corollary}
\newtheorem{conjecture}[theorem]{Conjecture}

\newenvironment{proof}{\noindent{\bf Proof\,}}{\hfill$\Box$}

\def\qed{\hfill $\Box$}

\begin{document}
\maketitle

\begin{abstract}
Let $\mc{K}_{3,3}^3$ be the $3$-graph with $15$ vertices $\{x_i, y_i: 1 \le i \le 3\}$ and $\{z_{ij}: 1 \le i,j \le 3\}$, and $11$ edges $\{x_1, x_2, x_3\}$, $\{y_1, y_2, y_3\}$ and $\{\{x_i, y_j, z_{ij}\}: 1 \le i,j \le 3\}$. We show that for large $n$, the unique largest $\mc{K}_{3,3}^3$-free $3$-graph on $n$ vertices is a balanced blow-up of the complete $3$-graph on $5$ vertices. Our proof uses the stability method and a result on lagrangians of intersecting families that has independent interest.
\end{abstract}

\section{Introduction}
\label{sec::intro}

The {\em Tur\'an number} $\ex(n,F)$ is the maximum number of edges in an $F$-free $r$-graph on $n$ vertices.%
\footnote{An {\em $r$-graph} (or {\em $r$-uniform hypergraph}) $G$ consists of a vertex set and an edge set, each edge being some $r$-set of vertices. We say $G$ is {\em $F$-free} if it does not have a (not necessarily induced) subgraph isomorphic to $F$.}
It is a long-standing open problem in Extremal Combinatorics to develop some understanding of these numbers for general $r$-graphs $F$.
For ordinary graphs ($r=2$) the picture is fairly complete, but for $r \ge 3$ there are very few known results. 
Tur\'an \cite{T61} posed the natural question of determining $\ex(n,F)$ when $F=K^r_t$ is a complete $r$-graph on $t$ vertices.
To date, no case with $t>r>2$ of this question has been solved, even asymptotically.
For a summary of progress on hypergraph Tur\'an problems before 2011 we refer the reader to the survey \cite{K}. 
Since then, most progress has been made by the computer-assisted method of Flag Algebras (see \cite{BT,FV}), although a new result by the method of link multigraphs was also obtained in \cite{KM}.

In this paper, we determine the Tur\'an number of the $3$-graph $\mc{K}_{3,3}^3$ with  vertices $\{x_i, y_i: 1 \le i \le 3\}$ and $\{z_{ij}: 1 \le i,j \le 3\}$, and edges $\{x_1, x_2, x_3\}$, $\{y_1, y_2, y_3\}$ and $\{\{x_i, y_j, z_{ij}\}: 1 \le i,j \le 3\}$. For an integer $n \geq 5$, let $T_5^3(n)$ denote the balanced blow-up of $K_5^3$ on $n$ vertices, that is, we partition the vertices into $5$ parts of sizes $\lfloor n/5 \rfloor$ or $\lceil n/5 \rceil$, and take as edges all triples in which the vertices belong to $3$ distinct parts. Write $t_5^3(n) := e(T_5^3(n))$. Our main result is as follows.

\begin{theorem} \label{th::main}
$\ex(n, \mc{K}_{3,3}^3) = t_5^3(n)$ for sufficiently large $n$. Moreover, if $n$ is sufficiently large and $G$ is a $\mc{K}_{3,3}^3$-free $3$-graph with $n$ vertices and $t_5^3(n)$ edges, then $G \cong T_5^3(n)$.
\end{theorem}

We prove Theorem \ref{th::main} by the stability method and lagrangians. Given an $r$-graph $G$ on $[n]=\{1,\dots,n\}$, we define a polynomial in the variables $x = (x_1,\dots,x_n)$ by
\[p_G(x):=\sum_{e \in E(G)} \prod_{i \in e} x_i.\]
The \emph{lagrangian} of $G$ is 
\[\lL(G) = \max \{ p_G(x): x_i \ge 0 \text{ for } 1 \le i \le n \text{ and } \sum_{i=1}^n x_i=1 \}.\]
A key tool in the proof will be the following result that determines the maximum possible lagrangian among all intersecting $3$-graphs: 
it is uniquely achieved by $K^3_5$, which has $\lL(K_5^3) = \binom{5}{3} (1/5)^3 = 2/25$.

\begin{theorem} \label{th::main3int}
Let $G$ be an intersecting $3$-graph. If $G \neq K_5^3$, then $\lL(G) \leq \lL(K_5^3) - 10^{-3}$.
\end{theorem}

We use the following notation and terminology throughout the paper. For a positive integer $n$ let $S_n$ denote the set of all permutations of $[n] := \{1, \ldots, n\}$. For a vector $x = (x_1, \ldots, x_n)$ of real numbers, the \emph{support} of $x$ is $Supp(x) := \{1 \leq i \leq n : x_i \neq 0\}$. For a set $A \sub [n]$ and a permutation $\pi \in S_n$ let $\pi(A) = \{\pi(a) : a \in A\}$. For a family $\mc{F}$ of subsets of $[n]$ let $\pi(\mc{F}) = \{\pi(A) : A \in \mc{F}\}$. A permutation $\pi \in S_n$ is an \emph{automorphism of $\mc{F}$} if $\pi(\mc{F}) = \mc{F}$. Let $\mc{F}$ be a family of subsets of $[n]$ and let $I \sub [n]$. We say that \emph{$\mc{F}$ covers pairs with respect to $I$} if for every $i,j \in I$ there exists some $F \in \mc{F}$ such that $\{i,j\} \sub F$. If $I = [n]$, then we will simply say that \emph{$\mc{F}$ covers pairs}. An intersecting $r$-graph $\mc{F}$ on $[n]$ is \emph{maximal (intersecting)} if for every $A \in {[n] \choose r} \sm \mc{F}$ there exists some $B \in \mc{F}$ such that $A \cap B = \es$. For a family $\mc{F}$ of subsets of $[n]$ and a set $I \sub [n]$ let $\mc{F}[I] = \{F \in \mc{F} : F \sub I\}$ be the restriction of $\mc{F}$ to $I$. For a family $\mc{F}$ of subsets of $[n]$ and a positive integer $r \leq n$ let $Gen(n, r, \mc{F}) = \{A \in {[n] \choose r} : \exists F \in \mc{F} \textrm{ such that } A \supseteq F\}$ denote the $r$-uniform family which is \emph{generated} by $\mc{F}$. Given an intersecting family $\mc{F}$ of subsets of $[n]$, a \emph{shift} \footnote{This definition of shifting used in this paper is different from the standard shifting technique introduced by Erd\H{o}s, Ko and Rado in~\cite{EKR}} of $\mc{F}$ is any family obtained by applying the following rule: as long as there exist $i \in A \in \mc{F}$ such that $\mc{F}' = (\mc{F} \sm \{A\}) \cup \{A \sm \{i\}\}$ is intersecting, replace $\mc{F}$ by $\mc{F}'$ and repeat. Note that, by definition, any shift of an intersecting family is also an intersecting family. By some abuse of notation, we denote any shift of $\mc{F}$ by $S(\mc{F})$. Given an $r$-graph $F$, the $t$-fold \emph{blow-up} $F(t)$ is obtained by replacing each vertex of $F$ by $t$ vertices, and each edge of $F$ by all $t^r$ edges of the complete $r$-partite $r$-graph spanned by the corresponding new vertices.

The organisation of this paper is as follows. In the next section we gather some simple properties of shifted families and optimal assignments for lagrangians, that will be used in Section 3 to prove Theorem \ref{th::main3int}. In Section 4 we prove Theorem \ref{th::main} via the stability method: first we prove Theorem \ref{prop::stability}, which gives the approximate structure of extremal examples, and then we refine this to give the exact result and uniqueness of structure. The final section contains some concluding remarks and open problems.

\section{Properties of shifted families and optimal assignments}
\label{sec::shifting}

We start with two simple observations.

\begin{observation} \label{obs::simple}
Let $\mc{F}_1$ and $\mc{F}_2$ be families of sets. If $\mc{F}_1 \sub \mc{F}_2$, then $\lL(\mc{F}_1) \leq \lL(\mc{F}_2)$.
\end{observation}

\begin{observation} \label{obs::support}
Let $\mc{F}$ be a family of subsets of $[n]$, let $t \leq n$ be a positive integer and let $I = \{i_1, \ldots, i_t\} \sub [n]$. Let $a_1, \ldots, a_n$ be non-negative real numbers such that $a_i = 0$ for every $i \notin I$. Then $p_\mc{F}(a_1, \ldots, a_n) = p_{\mc{F}[I]}(a_{i_1}, \ldots, a_{i_t})$.
\end{observation}

The following lemma was proved in~\cite{FR} (using slightly different terminology).
\begin{lemma} \label{lem::FR}
Let $\mc{F}$ be a family of $r$-subsets of $[n]$. Suppose $a_1, \dots, a_n \ge 0$ such that $\sum_{i=1}^n a_i = 1$, $\lL(\mc{F}) = p_\mc{F}(a)$ and $I = Supp(a)$ is minimal. Then $\mc{F}[I]$ covers pairs with respect to $I$.
\end{lemma}

We deduce that it suffices to consider intersecting $r$-graphs that cover pairs; more precisely, we have the following statement.
\begin{lemma} \label{lem::maxCoverPairs}
For positive integers $r \leq n$, let $m_1$ denote the maximum of $\lL(\mc{F})$ over all intersecting $r$-graphs $\mc{F}$ on $[n]$ and let $m_2$ denote the maximum of $\lL(\mc{F})$ over all intersecting $r$-graphs $\mc{F}$ on $[n]$ that cover pairs with respect to $\bigcup_{F \in \mc{F}} F$. Then $m_1 = m_2$.
\end{lemma}

\begin{proof}
It is obvious that $m_1 \geq m_2$; hence it suffices to prove that $m_1 \leq m_2$. Let $\mc{F}$ be an intersecting $r$-graph on $[n]$ such that $\lL(\mc{F}) = m_1$. Let $a_1, \dots, a_n \ge 0$ such that $\sum_{i=1}^n a_i = 1$, $\lL(\mc{F}) = p_\mc{F}(a)$ and $A = Supp(a)$ is minimal. Then $\mc{F}[A]$ covers pairs with respect to $A = \bigcup_{F \in \mc{F}[A]} F$ by Lemma~\ref{lem::FR}. Moreover, $\lL(\mc{F}) = \lL(\mc{F}[A])$ holds by Observations~\ref{obs::simple} and~\ref{obs::support}. The lemma follows since clearly $\mc{F}[A]$ is an intersecting $r$-graph on $[n]$ and $\lL(\mc{F}[A]) \leq m_2$.
\end{proof}

Next we need some properties of the shifted family $S(\mc{F})$ defined in the introduction.

\begin{lemma} \label{lem::uniqueIntersection}
Let $\mc{F}$ be an intersecting family. Then for every $i \in A \in S(\mc{F})$ there exists $B \in S(\mc{F})$ such that $A \cap B = \{i\}$.
\end{lemma}

\begin{proof}
Suppose there exists some $i \in A \in S(\mc{F})$ such that $(A \sm \{i\}) \cap B \neq \es$ for every $B \in S(\mc{F})$. Then $(S(\mc{F}) \sm \{A\}) \cup \{A \sm \{i\}\}$ is intersecting, contrary to $S(\mc{F})$ being shifted.
\end{proof}

Combining Lemma \ref{lem::uniqueIntersection} with the following observation, we see that $S(\mc{F})$ is an antichain.

\begin{observation} \label{rem::StrictContainment}
Let $\mc{F}$ be an intersecting family and suppose that for every $i \in A \in \mc{F}$ there exists $B \in \mc{F}$ such that 
$A \cap B = \{i\}$. Then $\mc{F}$ does not contain two sets $S$ and $T$ such that $S \subn T$. 
\end{observation}

Next we show that a maximal intersecting $r$-graph is generated by its shifted family.

\begin{lemma} \label{lem::genShift}
Let $\mc{F}$ be a maximal intersecting $r$-graph on $[n]$. Then $Gen(n, r, S(\mc{F})) = \mc{F}$.
\end{lemma}

\begin{proof}
Since $\mc{F}$ is maximal, it suffices to show $\mc{F} \sub Gen(n, r, S(\mc{F}))$. For any $A \in \mc{F}$, by construction of $S(\mc{F})$, there must exist some $B \in S(\mc{F})$ such that $B \sub A$; then $A \in Gen(n, r, S(\mc{F}))$ by definition.
\end{proof}

Next we show how to exploit symmetries of a family when computing its lagrangian.

\begin{lemma} \label{lem::transposition}
Let $\mc{F}$ be a family of subsets of $[n]$. Suppose for some $1 \leq i < j \leq n$ that the transposition $(ij)$ is an automorphism of $\mc{F}$. Suppose $a_1, \dots, a_n \ge 0$ and $\sum_{i=1}^n a_i = 1$. Let $a'_i = a'_j = (a_i + a_j)/2$, and for every $t \in [n] \sm \{i,j\}$ let $a'_t = a_t$. Then $p_\mc{F}(a') \geq p_\mc{F}(a)$.
\end{lemma}

\begin{proof}
Since $(ij)$ is an automorphism of $\mc{F}$, it is easy to see that
\begin{align*}
p_\mc{F}(a') - p_\mc{F}(a) = \sum_{\stackrel{F \in \mc{F}}{\{i,j\} \subseteq F}} \left((a_i + a_j)^2/4 - a_i a_j\right) \prod_{t \in F \sm \{i,j\}} a_t \geq 0.
\end{align*}
\end{proof}

\begin{corollary} \label{cor::sym}
Let $\mc{F}$ be a family of subsets of $[n]$. Let $\mc{P}_\mc{F}$ be the partition of $[n]$ into equivalence classes of the relation in which $i \sim j$ if and only if $(ij)$ is an automorphism of $\mc{F}$. Then there is an assignment $a = (a_1,\dots,a_n)$ with  $a_1, \dots, a_n \ge 0$, $\sum_{i=1}^n a_i = 1$ and $p_\mc{F}(a) = \lL(\mc{F})$ such that $a$ is constant on each part of $\mc{P}_\mc{F}$.
\end{corollary}

\begin{proof}
Given an assignment $a$ and $P \in \mc{P}_\mc{F}$, let $\bar{a}_P = |P|^{-1} \sum_{i \in P} a_i$. Consider an assignment $a = (a_1,\dots,a_n)$ with  $a_1, \dots, a_n \ge 0$, $\sum_{i=1}^n a_i = 1$ and $p_\mc{F}(a) = \lL(\mc{F})$ that minimises $S(a) := \sum_{P \in \mc{P}_\mc{F}} \sum_{i \in P} |a_i - \bar{a}_P|$. We claim $S(a)=0$, i.e.\ $a$ is constant on each part of $\mc{P}_\mc{F}$. For suppose not, and consider some $P \in \mc{P}_\mc{F}$ and $i,j$ in $P$ with $a_i > \bar{a}_p > a_j$. Define $a'$ as in Lemma \ref{lem::transposition}; then $p_\mc{F}(a') = \lL(\mc{F})$ and $S(a')<S(a)$, contradicting the choice of $a$.
\end{proof}

\medskip

Given a family $\mc{F}$ of subsets of $[n]$ and $1 \leq i \leq n$, let $\mc{F}_i = \{F \in \mc{F} : i \in F\}$ and $\mc{F}_i^- = \{F \sm \{i\} : F \in \mc{F}_i\}$. The next lemma shows that, in an optimal assignment to $p_\mc{F}$, we do not need to assign any weight to a vertex $j$, if there is another vertex $i$ that \emph{dominates} $j$ for $\mc{F}$, in that $\mc{F}_j^- \sub \mc{F}_i^-$.

\begin{lemma} \label{lem::dominate}
Let $\mc{F}$ be a family of subsets of $[n]$. Suppose that $1 \leq i < j \leq n$ are such that $i$ dominates $j$ for $\mc{F}$. Suppose $a_1, \dots, a_n \ge 0$ and $\sum_{i=1}^n a_i = 1$. Let $a'_j = 0$, let $a'_i = a_i + a_j$ and for every $t \in [n] \sm \{i,j\}$ let $a'_t = a_t$. Then $p_\mc{F}(a') \geq p_\mc{F}(a)$.
\end{lemma}

\begin{proof}
Since $i$ dominates $j$ for $\mc{F}$, there is no $F \in \mc{F}$ such that $\{i,j\} \sub F$. It thus follows that
\begin{align*}
p_\mc{F}(a') - p_\mc{F}(a) 
& = \sum_{F \in \mc{F}_i} (a'_i - a_i) \prod_{t \in F \sm \{i\}} a_t + \sum_{F \in \mc{F}_j} (a'_j - a_j) \prod_{t \in F \sm \{j\}} a_t\\
& = \sum_{F \in \mc{F}_i^- \sm \mc{F}_j^-} a_j \prod_{t \in F} a_t \geq 0,
\end{align*}
where the second equality follows by the assumption that $\mc{F}_j^- \sub \mc{F}_i^-$.
\end{proof}

Finally, we need a property of optimal assignments that follows from the theory of Lagrange multipliers (see \cite[Theorem 2.1]{FR}).

\begin{lemma} \label{lem::link}
Suppose $G$ is an $r$-graph on $[n]$. Suppose $a_1, \dots, a_n \ge 0$ with $\sum_{i=1}^n a_i = 1$ and $p_G(a) = \lL(G)$ is such that $Supp(a)$ is minimal. Then $\frac{\pl p_G(x)}{\pl x_i}(a) = r\lL(G)$ for all $i \in Supp(a)$.
\end{lemma}

\section{Lagrangians of intersecting $3$-graphs}
\label{sec::3uniform}

In this section we prove Theorem \ref{th::main3int}, by showing that the maximum possible lagrangian among all intersecting $3$-graphs is uniquely achieved by $K^3_5$. Note that by Observation \ref{obs::simple} and Lemma \ref{lem::maxCoverPairs} it suffices to consider maximal intersecting $3$-graphs that cover pairs. Our proof consists of classifying such $3$-graphs by their shifted families, and verifying the required bound in each case. We will prove that it suffices to consider the shifted families illustrated in Figure 1.

\begin{figure}\label{fig}
\begin{center}
\includegraphics[scale=0.85]{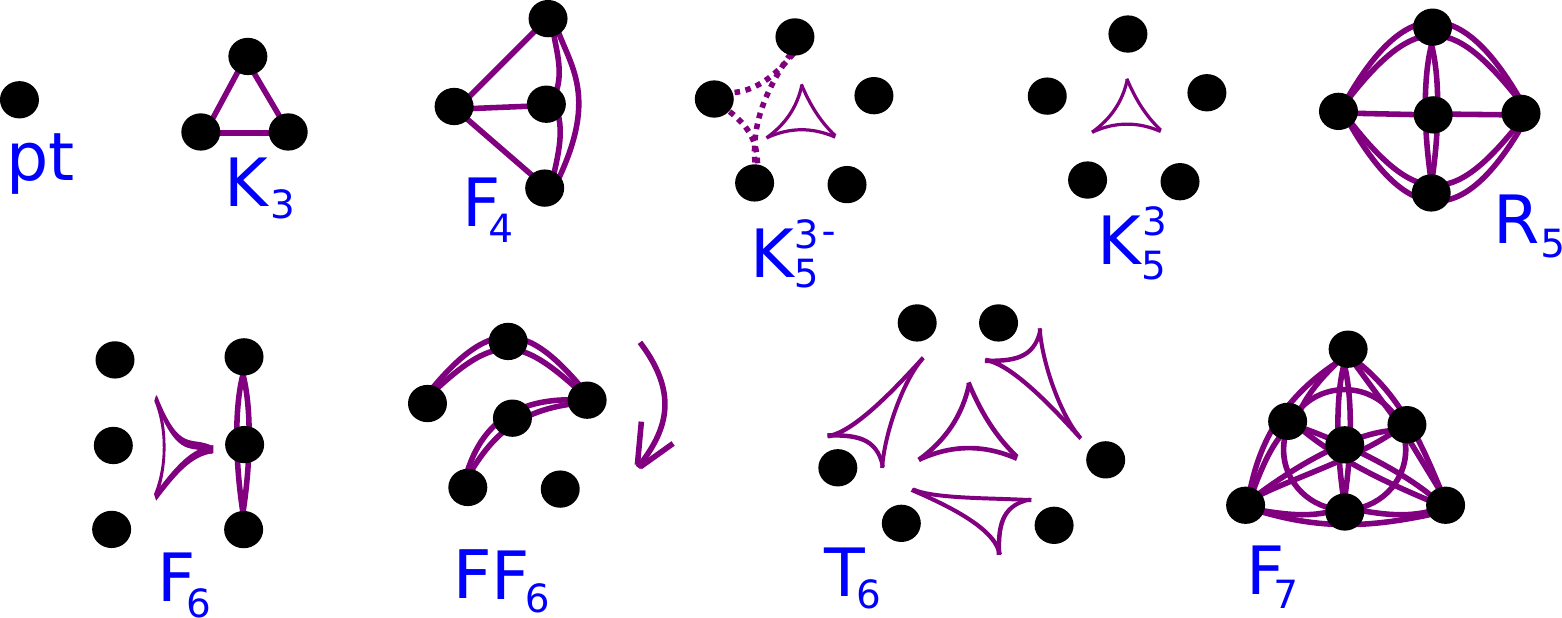}
\end{center}
\caption{Shifted families}
\end{figure}

We start with the second part of the proof, in which we compute or estimate the lagrangians of these families.
First we consider $K^3_5$, the complete $3$-graph on $5$ vertices.

\begin{lemma} \label{lem::K53}
$\lL(K_5^3) = \frac{2}{25}$.
\end{lemma}

\begin{proof}
Since every $\pi \in S_5$ is an automorphism of $K_5^3$, we have $\lL(K_5^3) = p_{K_5^3}(1/5, \ldots, 1/5) = \frac{2}{25}$ by Corollary~\ref{cor::sym}.
\end{proof}

Next we consider the Fano plane $F_7$.

\begin{lemma} \label{lem::F7}
$\lL(F_7) = \frac{1}{27}$.
\end{lemma}

\begin{proof}
Note first that $\lL(F_7) \geq \frac{1}{27}$ as this value is obtained by splitting the entire weight equally among any 3 vertices which form an edge. Suppose for a contradiction that $\lL(F_7) > \frac{1}{27}$. Let $a_1, \dots, a_7 \ge 0$ with $\sum_{i=1}^n a_i = 1$ and $p_{F_7}(a) = \lL(F_7)$. We can assume that $a$ has full support, i.e. $a_i>0$ for all $1 \le i \le 7$. For otherwise, $F_7[Supp(a)]$ is contained in the $2$-blow-up of an edge, which has the same lagrangian as an edge, namely $(1/3)^3 = 1/27$. Note that for every pair $\{j,k\}$ there is a unique $i$ such that $\{i,j,k\} \in E(F_7)$. By Lemma \ref{lem::link}, we can choose $a$ such that $\frac{\pl p_{F_7}(x)}{\pl x_i}(a) = 3\lL(F_7)$ for all $1 \le i \le 7$ (note that $a$ has minimal support by assumption). Summing over $i$ we obtain $\sum_{1 \le j<k \le 7} a_ja_k = 21 \lL(F_7)$. By symmetry we deduce that $21 \lL(F_7) \le \binom{7}{2}(1/7)^2$, i.e.\ $\lL(F_7) \le \frac{1}{49}$ which is a clear contradiction. 
\end{proof}

Next we consider the star, where the generating family is a single point, denoted $pt$.

\begin{lemma} \label{lem::star}
For an integer $n \geq 3$ let $\mc{F} = Gen(n,3,pt) = \{\{1,i,j\} : 2 \leq i < j \leq n\}$. Then $\lL(\mc{F}) < \frac{2}{27}$.
\end{lemma}

\begin{proof}
By Corollary~\ref{cor::sym} there is $0 \leq x \leq 1$ such that 
\[\lL(\mc{F}) = p_\mc{F}(1-x, \tfrac{x}{n-1}, \ldots, \tfrac{x}{n-1}) 
= {n-1 \choose 2} (1-x) \left(\frac{x}{n-1}\right)^2 = \frac{n-2}{2(n-1)} x^2(1-x).\]
The maximum occurs at $x=2/3$, so $\lL(\mc{F}) \le \frac{n-2}{n-1} \cdot 2/27 < 2/27$.
\end{proof}

Next we consider the family generated by the triangle $K_3$.

\begin{lemma} \label{lem::triangle}
For an integer $n \geq 3$ let 
\[\mc{F} = Gen(n,3,K_3) = \{\{i,j,t\} : 1 \leq i < j \leq 3, t \in [n] \sm \{i,j\}\}.\] 
Then $\lL(\mc{F}) = \frac{1}{16}$.
\end{lemma}

\begin{proof}
By Corollary~\ref{cor::sym} there is $0 \leq x \leq 1/3$ such that 
\[\lL(\mc{F}) = p_\mc{F}(x, x, x, \tfrac{1-3x}{n-3}, \ldots, \tfrac{1-3x}{n-3}) 
= x^3 + 3(n-3) x^2 \cdot \frac{1-3x}{n-3} = 3x^2 - 8x^3.\]
We can assume $0 < x < 1/3$, as $3x^2 - 8x^3$ is $0$ at $x=0$ and $1/27$ at $x=1/3$. 
Differentiating, we see that the maximum occurs at $x=1/4$, so $\lL(\mc{F}) = 1/16$.
\end{proof}

Now we consider incomplete $3$-graphs on $5$ vertices, which are dominated by the $3$-graph $K_5^{3-}$ where one edge is missing.

\begin{lemma} \label{lem::AtMost5Vertices}
Let $\mc{F}$ be a $3$-graph on $[5]$. If $\mc{F} \neq K_5^3$, then $\lL(\mc{F}) \leq \lL(K_5^{3-}) \le \lL(K_5^3) - 10^{-3}$.
\end{lemma}

\begin{proof}
By Observation~\ref{obs::simple} it suffices to prove that $\lL(K_5^{3-}) \leq \frac{2}{25} - 10^{-3}$. By Corollary~\ref{cor::sym} there is $0 \leq x \leq 1/3$ such that 
\[\lL(K_5^{3-}) = p_{K_5^{3-}}(x, x, x, \tfrac{1-3x}{2}, \tfrac{1-3x}{2}) 
= 6 x^2 \tfrac{1-3x}{2} + 3 x \tfrac{(1-3x)^2}{4} = \tfrac{3}{4} (x - 2 x^2 - 3 x^3).\]
We can assume $0 < x < 1/3$, as $f(x) := x - 2 x^2 - 3 x^3$ is $0$ at $x=0$ and $x=1/3$.
Differentiating, we see that the maximum occurs at $x = \tfrac{\sqrt{13} - 2}{9}$, 
so $\lL(K_5^{3-}) = \tfrac{13\sqrt{13}-35}{162} < \lL(K_5^3) - 10^{-3}$.
\end{proof}

Next we consider the family generated by $F_4 = \{ \{1,2\}, \{1,3\}, \{1,4\}, \{2,3,4\} \}$.

\begin{lemma} \label{lem::blockedStar}
For an integer $n \geq 4$ let 
\[\mc{F} = Gen(n,3,F_4) = \{\{2,3,4\}\} \cup \{\{1,i,j\} : 2 \leq i \leq 4, 2 \leq j \neq i \leq n\}.\]
Then $\lL(\mc{F}) \leq \lL(K_5^3) - 10^{-3}$.
\end{lemma}

\begin{proof}
For every $6 \leq i \leq n$ we have $\mc{F}_i^- = \{\{1,2\}, \{1,3\}, \{1,4\}\} = \mc{F}_5^-$. By Lemma~\ref{lem::dominate} there are $a_1, \ldots, a_5 \ge 0$ such that $\sum_{i=1}^5 a_i = 1$ and $\lL(\mc{F}) = p_\mc{F}(a_1, \ldots, a_5, 0, \ldots, 0)$. Writing $\mc{F}' = \mc{F}[\{1, \ldots, 5\}]$, we have $\lL(\mc{F}) = p_{\mc{F}'}(a)$ by  Observation~\ref{obs::support}. Since $\{2,3,5\} \notin \mc{F}$, Lemma~\ref{lem::AtMost5Vertices} now gives $\lL(\mc{F}) \leq \lL(K_5^3) - 10^{-3}$.
\end{proof}

Now we consider the family generated by $R_5 = \{ \{1,2\}, \{1,4\}, \{1,3,5\}, \{2,3,4\}, \{2,4,5\} \}$.

\begin{lemma} \label{lem::5VertexGenFam}
For an integer $n \geq 5$ let 
\[\mc{F} = Gen(n,3,R_5) = \{\{1,3,5\}, \{2,3,4\}, \{2,4,5\}\} \cup \{\{1,i,j\} : i \in \{2,4\}, 2 \leq j \neq i \leq n\}.\]
Then $\lL(\mc{F}) \leq \lL(K_5^3) - 10^{-3}$. 
\end{lemma}

\begin{proof}
For every $6 \leq i \leq n$ we have $\mc{F}_i^- = \{\{1,2\}, \{1,4\}\} \sub \{\{1,2\}, \{1,4\}, \{1,3\}, \{2,4\}\} = \mc{F}_5^-$. By Lemma~\ref{lem::dominate} and Observation~\ref{obs::support} there are $a_1, \ldots, a_5 \ge 0$ such that $\sum_{i=1}^5 a_i = 1$ and $\lL(\mc{F}) = p_\mc{F}(a_1, \ldots, a_5, 0, \ldots, 0) = p_{\mc{F}'}(a)$, where  $\mc{F}' = \mc{F}[\{1, \ldots, 5\}]$. Since $\{2,3,5\} \notin \mc{F}$, Lemma~\ref{lem::AtMost5Vertices} gives $\lL(\mc{F}) \leq \lL(K_5^3) - 10^{-3}$.
\end{proof}

The next few lemmas are in preparation for Lemma \ref{lem::AtMost6Vertices}, which shows the required bound for any intersecting $3$-graph on $[6]$ that covers pairs. First we show that if some pair belongs to $4$ edges then we have a star.

\begin{lemma} \label{lem::existsPairCovered4times}
Let $\mc{F}$ be an intersecting $3$-graph on $[6]$ that covers pairs.
Suppose that there exist $1 \leq i < j \leq 6$ such that $\{i, j, k\} \in \mc{F}$ for every $k \in [6] \sm \{i, j\}$. 
Then $\mc{F}=Gen(6,3,pt)$, so $\lL(\mc{F}) < \frac{2}{27}$.
\end{lemma}

\begin{proof}
Assume without loss of generality that $i=1$ and $j=2$ are such indices, that is, $\{1, 2, 3\}, \{1, 2, 4\}, \{1, 2, 5\}, \{1, 2, 6\} \in \mc{F}$. Since $\mc{F}$ covers pairs, there exists some $F_{34} \in \mc{F}$ such that $\{3,4\} \sub F_{34}$. Since $\mc{F}$ is intersecting, $F_{34} \cap \{1,2,5\} \neq \es$ and $F_{34} \cap \{1,2,6\} \neq \es$. Hence $F_{34} = \{1,3,4\}$ or $F_{34} = \{2,3,4\}$; assume without loss of generality that $F_{34} = \{1,3,4\}$. Similarly, for every $3 \leq i < j \leq 6$ there exists some $F_{ij} \in \mc{F}$ such that $\{i,j\} \sub F_{ij}$. Since $\mc{F}$ is intersecting, $F_{ij} \cap \{1,3,4\} \neq \es$ and $F_{ij} \cap \{1,2,k\} \neq \es$ for every $3 \leq k \leq 6$. Hence $F_{ij} = \{1,i,j\}$ for every $3 \leq i < j \leq 6$. Since $\mc{F}$ is intersecting, $\mc{F} = \{\{1,i,j\} : 2 \leq i < j \leq 6\} = Gen(6,3,pt)$. Thus $\lL(\mc{F}) < \frac{2}{27}$ by Lemma~\ref{lem::star}. 
\end{proof}

Next we consider the family $T_6$. This is the Tur\'an construction for a $K^3_4$-free graph on $6$ vertices: it has $3$ parts of size $2$, all triples with one vertex in each part, and all 2-1 triples according to some cyclic order. Note that this is not intersecting, but contains several intersecting families that arise in the analysis.

\begin{lemma} \label{lem::bigFamily}
$\lL(T_6) \leq \lL(K_5^3) - 10^{-3}$. 
\end{lemma}

\begin{proof}
First we note that it suffices to consider assignments $a=(a_1,\dots,a_6)$ to $p_{T_6}(x)$ with full support. For if $a_i=0$ for some $i$ then, since there exists $A \in \binom{[6]}{3}$ such that $i \in A \notin T_6$, the required bound holds by Observation~\ref{obs::support} and Lemma~\ref{lem::AtMost5Vertices}. Next note that for each of the $3$ parts of $T_6$, the transposition interchanging the $2$ vertices of the part is an automorphism of $T_6$. Then by Corollary~\ref{cor::sym} there are $x, y, z \geq 0$ such that $x + y + z = 1/2$ and $\lL(T_6) = 2 x y^2 + 2 x^2 z + 2 y z^2 + 8 x y z$. Also, by Lemma \ref{lem::link} we have 
\[ 3\lL(T_6) = y^2 + 2zx + 4yz = z^2 + 2xy + 4zx = x^2 + 2yz + 4xy.\]
Summing we obtain 
\[ 9\lL(T_6) = x^2+y^2+z^2 + 6(xy+yz+xz) = (x+y+z)^2 + 4(xy+yz+xz) \le (1/2)^2 + 4 \cdot 3 \cdot (1/6)^2 = 7/12. \]
Thus $\lL(T_6) \le 7/108 < 2/25 - 10^{-3}$, as required.
\end{proof}

The next case to consider is $F_6$. This is the $3$-graph on $[6]$ where the edges are $\{1,2,3\}$ and all $9$ triples with exactly one vertex in $\{1,2,3\}$.

\begin{lemma} \label{lem::smallFamily}
$\lL(F_6) \leq \lL(K_5^3) - 10^{-3}$. 
\end{lemma}

\begin{proof}
Note that both of the $2$ parts of $F_6$ consist of $3$ vertices that are equivalent under automorphisms of $F_6$. By Corollary~\ref{cor::sym} there exists $0 \leq x \leq 1/3$ such that 
\[ \lL(F_6) = p_{F_6}(x,x,x,(1-3x)/3,(1-3x)/3,(1-3x)/3) = x^3 + 9x(1/3-x)^2 = 10x^3 - 6x^2 + x.\]
Since $f(x) := 10x^3 - 6x^2 + x$ is $0$ at $x=0$ and $1/27$ at $x=1/3$ we can assume that $0<x<1/3$.
Differentiating, we see that the maximum occurs at $x = \tfrac{6 \pm \sqrt{6}}{30}$.
Then we calculate $f(\tfrac{6 \pm \sqrt{6}}{30}) = \tfrac{9 \mp \sqrt{6}}{225} < 2/25 - 10^{-3}$.
\end{proof}

The final preparatory lemma shows that if every pair belongs to precisely $2$ edges of $\mc{F}$ then $\mc{F} = FF_6$ and the required bound holds. Here $FF_6$ is the $3$-graph from \cite{FF2}; it has $10$ edges, which are given by the orbit of the $2$ shown in Figure 1 under the $5$-fold rotation symmetries. Namely, its edge-set is $\{\{i, i + 1, i + 2\}, \{i, i + 2, 6\} : i \in [5]\}$, where we identify each $i \in [5]$ with its residue modulo $5$, and addition is to be understood modulo $5$.
 
\begin{lemma} \label{lem::everyPairCoveredTwice}
Let $\mc{F}$ be an intersecting $3$-graph on $[6]$.
If $|\{A \in \mc{F} : \{i,j\} \sub A\}| = 2$ holds for every $1 \leq i < j \leq 6$, 
then $\mc{F} = FF_6$ and $\lL(\mc{F}) \leq \lL(K_5^3) - 10^{-3}$.
\end{lemma}

\begin{proof}
Assume without loss of generality that $\{1, 2, 3\}, \{1, 2, 4\} \in \mc{F}$ but $\{1, 2, 5\}, \{1, 2, 6\} \notin \mc{F}$. Note that $|\mc{F}| = \frac{1}{3} \sum_{1 \le i < j \le 6} |\{A \in \mc{F} : \{i,j\} \sub A\}| = 10$, so $\mc{F}$ is maximal intersecting. Thus $|\mc{F} \cap \{A, [6] \sm A\}| = 1$ for every $A \in \binom{[6]}{3}$. Hence $\{3, 4, 6\}, \{3, 4, 5\} \in \mc{F}$. By assumption, there are two sets $\{i, 5, 6\}, \{j, 5, 6\} \in \mc{F}$. Since $\mc{F}$ is intersecting, $\{1, 2, a\} \cap \{b, 5, 6\} \neq \es$ for every $a \in \{3,4\}$ and $b \in \{i,j\}$. Hence $\{1, 5, 6\}, \{2, 5, 6\} \in \mc{F}$. To cover the pair $\{1,3\}$ exactly twice, we must have $|\mc{F} \cap \{\{1,3,5\}, \{1,3,6\}\}| = 1$; assume that $\{1, 3, 5\} \in \mc{F}$ but $\{1, 3, 6\} \notin \mc{F}$, so $\{2, 4, 5\} \in \mc{F}$ (the complementary case yields an isomorphic family). Since $\{1, 2, 6\}, \{1, 3, 6\} \notin \mc{F}$, to cover $\{1,6\}$ twice we have $\{1, 4, 6\} \in \mc{F}$. Finally, to cover $\{2,3\}$ twice we have $\{2, 3, 6\} \in \mc{F}$, so $\mc{F} = FF_6$ (the permutation $(1)(2)(354)(6)$ is an appropriate isomorphism). Since $FF_6 \sub T_6$ the required bound on $\lL(\mc{F})$ holds by Observation~\ref{obs::simple} and Lemma~\ref{lem::bigFamily}.
\end{proof}

Now we combine the previous lemmas to analyse all intersecting $3$-graphs on $[6]$ that cover pairs.

\begin{lemma} \label{lem::AtMost6Vertices}
Let $\mc{F}$ be an intersecting $3$-graph on $[6]$ that covers pairs. 
Then $\lL(\mc{F}) \leq \lL(K_5^3) - 10^{-3}$.
\end{lemma}

\begin{proof}
By Observation~\ref{obs::simple} we can assume $\mc{F}$ is maximal. Hence, for every $A \in \binom{[6]}{3}$ we have $|\mc{F} \cap \{A, [6] \sm A\}| = 1$, so $|\mc{F}| = 10$. We claim that $\mc{F}$ is $K^3_4$-free. For suppose that $\mc{F}$ contains all triples in $[4]$. Since $\mc{F}$ covers $56$, without loss of generality $\{1,5,6\} \in \mc{F}$, but this is disjoint from $\{2,3,4\}$, a contradiction. Next we note that by Lemma~\ref{lem::existsPairCovered4times} we can assume there is no pair covered by $4$ edges. Also, by Lemma~\ref{lem::everyPairCoveredTwice} we can assume that not every pair is covered exactly twice. Since $|\mc{F}| = 10$, some pair is covered by exactly $3$ sets of $\mc{F}$. Without loss of generality $\{1, 2, 3\}$, $\{1, 2, 4\}$ and $\{1, 2, 5\}$ are in $\mc{F}$, but $\{1, 2, 6\} \notin \mc{F}$, so $\{3, 4, 5\} \in \mc{F}$. By Lemma \ref{lem::smallFamily} we can assume that $\mc{F} \ne F_6$, so $\mc{F}$ does not contain all triples with one vertex in $\{3, 4, 5\}$. Without loss of generality $\{1,5,6\} \notin \mc{F}$, so $\{2,3,4\} \in \mc{F}$. Since $\mc{F}$ is $K^3_4$-free, $\{1,3,4\} \notin \mc{F}$, so $\{2,5,6\} \in \mc{F}$. 

So far, $\mc{F}$ contains $\{1,2,3\}$, $\{1,2,4\}$, $\{1,2,5\}$, $\{3,4,5\}$, $\{2,3,4\}$ and $\{2,5,6\}$. Now we claim that $\mc{F} \sub T_6$; the bound on $\lL(\mc{F})$ will then follow from Lemma \ref{lem::bigFamily}. Consider any partition of $[6]$ into cyclically ordered pairs $(A_1,A_2,A_3)$. We can divide the $20$ triples in $[6]$ into two self-complementary families: the $8$  triples with one vertex in each part, and the $12$ triples with two vertices in one part and one in another. To show that $\mc{F} \sub T_6$ with this partition, it suffices to show that $\mc{F}$ contains all triples $a_ia'_ia_{i+1}$ with $a_i,a'_i \in A_i$ and $a_{i+1} \in A_{i+1}$ for $1 \le i \le 3$ (where $A_4:=A_1$). Indeed, since $\mc{F}$ is intersecting, it would then follow that $\mc{F}$ cannot contain any triplet which is not in $T_6$. 

Since $\mc{F}$ covers $\{1,6\}$ it contains $\{1,3,6\}$ or $\{1,4,6\}$. If $\mc{F}$ contains both, then it is contained in $T_6$ with parts $\{1,6\}$, $\{3,4\}$, $\{2,5\}$. Otherwise, since the transposition $(34)$ is an automorphism of the family constructed so far, we can assume that $\{1,4,6\} \in \mc{F}$ and $\{1,3,6\} \notin \mc{F}$, so $\{2,4,5\} \in \mc{F}$. Since $\mc{F}$ is $K^3_4$-free, $\{1,4,5\} \notin \mc{F}$, so $\{2,3,6\} \in \mc{F}$. Since $\{2,4\}$ is not contained in $4$ edges, $\{2,4,6\} \notin \mc{F}$, so $\{1,3,5\} \in \mc{F}$. Thus $\mc{F}$ is contained in $T_6$ with parts $\{1,4\}$, $\{2,6\}$, $\{3,5\}$.
\end{proof}

Now that we have estimated the lagrangians of the families in Figure 1, we turn to the first part of the proof, namely showing that it suffices to consider these families. In the next lemma we show that we can restrict to $3$-graphs on at most $7$ vertices.

\begin{lemma} \label{lem::genFamilyIsSmall}
Let $\mc{F}$ be an intersecting $3$-graph on $[n]$ that covers pairs, such that for every $i \in A \in \mc{F}$ there is $B \in \mc{F}$ such that $A \cap B = \{i\}$. Then $n \leq 7$.
\end{lemma}

\begin{proof}
Suppose for the sake of contradiction that $n \geq 8$. Assume without loss of generality that $\{1,2,3\} \in \mc{F}$. Since $\mc{F}$ covers pairs, for every $5 \leq i \leq 8$ there is some $F_i \in \mc{F}$ such that $\{4,i\} \sub F_i$. Since $\mc{F}$ is intersecting, $\{1,2,3\} \cap F_i \neq \es$ for every $5 \leq i \leq 8$. By the pigeonhole principle there exist $5 \leq i < j \leq 8$ and $1 \leq t \leq 3$ such that $t \in F_i \cap F_j$. Without loss of generality $\{1,4,5\}, \{1,4,6\} \in \mc{F}$. Considering $5 \in \{1,4,5\} \in \mc{F}$, by assumption there is some $A \in \mc{F}$ such that $\{1,4,5\} \cap A = \{5\}$. Since $\mc{F}$ is intersecting, $A \cap \{1,4,6\} \neq \es$, and so $6 \in A$. Since  $1 \notin A$ and $A \cap \{1,2,3\} \neq \es$, either $A = \{2,5,6\}$ or $A = \{3,5,6\}$. Since $\mc{F}$ covers pairs, there exists some $B \in \mc{F}$ such that $\{7,8\} \sub B$. Since $B \cap \{1,2,3\} \neq \es$ and $B \cap \{1,4,5\} \neq \es$, it follows that $B = \{1,7,8\}$. But then $A \cap B = \es$, contradiction.
\end{proof}

Next we show that if $\mc{F}$ has $7$ vertices then it must be the Fano plane.

\begin{lemma} \label{lem::F7IsUnique}
Let $\mc{F}$ be an intersecting $3$-graph on $[7]$ that covers pairs, such that for every $i \in A \in \mc{F}$ there is $B \in \mc{F}$ such that $A \cap B = \{i\}$. Then $\mc{F} = F_7$.
\end{lemma}

\begin{proof}
Assume without loss of generality that $\{1,2,3\} \in \mc{F}$. By assumption, for every $i \in \{1,2,3\}$ there exists some $F_i \in \mc{F}$ such that $F_i \cap \{1,2,3\} = \{i\}$. Also, since $\mc{F}$ covers pairs, for each pair $\{i,j\}$ we can fix $F_{ij} \in \mc{F}$ such that $\{i,j\} \sub F_{ij}$. Without loss of generality $F_1 = \{1,4,5\}$. Since $F_{67} \cap F_1 \neq \es$ and $F_{67} \cap \{1,2,3\} \neq \es$ we have $F_{67} = \{1,6,7\}$. Moreover, since $\mc{F}$ is intersecting, $\mc{F}$ does not contain $\{2,4,5\}$ or $\{2,6,7\}$. Hence, we can assume without loss of generality that $F_2 = \{2,4,6\}$.  Since $F_{57} \cap F_2 \neq \es$ and  $F_{57} \cap \{1,2,3\} \neq \es$ we have $F_{57} = \{2,5,7\}$. Moreover, since $\mc{F}$ is intersecting, $\mc{F}$ does not contain any of $\{3,4,5\}$, $\{3,6,7\}$, $\{3,4,6\}$ or $\{3,5,7\}$. Hence, without loss of generality $F_3 = \{3,4,7\}$. Since $F_{56} \cap F_3 \neq \es$ and $F_{56} \cap \{1,2,3\} \neq \es$ we have $F_{56} = \{3,5,6\}$. Now $\{\{1,2,3\}, F_1, F_2, F_3, F_{67}, F_{57}, F_{56}\}$ forms a copy of $F_7$. This is maximal intersecting, so $\mc{F} = F_7$.
\end{proof}

Our final lemma considers the case when the generating family is non-uniform.

\begin{lemma} \label{lem::Size2And3}
Let $\mc{F}$ be an intersecting family of subsets of $[n]$. Assume that
\begin{enumerate}[(i)]
\item $2 \leq |F| \leq 3$ for every $F \in \mc{F}$,
\item there exist sets $A,B \in \mc{F}$ such that $|A|=2$ and $|B|=3$,
\item for every $A \in \mc{F}$ and for every $i \in A$ there exists $B \in \mc{F}$ such that $A \cap B = \{i\}$,
\item $G := Gen(n, 3, \mc{F})$ covers pairs.
\end{enumerate}
Then $\lL(G) \leq \lL(K_5^3) - 10^{-3}$.
\end{lemma}

\begin{proof}
First note by (iii) and Observation~\ref{rem::StrictContainment} that $\mc{F}$ does not contain two sets $S$ and $T$ such that $S \subn T$.
Assume without loss of generality that $A = \{1,2\}$. Then without loss of generality $B = \{2,3,4\}$. Now we consider cases according to how many of $\{1,3\}$ and $\{1,4\}$ belong to $\mc{F}$. 

Suppose first that both $\{1,3\} \in \mc{F}$ and $\{1,4\} \in \mc{F}$. Then $\mc{F}$ cannot contain any other set, as this would have to contain one of $\{1,2\}$, $\{1,3\}$, $\{1,4\}$ in order to intersect all sets in $\mc{F}$. Thus $\mc{F}=F_4$, so $\lL(G) \leq \lL(K_5^3) - 10^{-3}$ by Lemma~\ref{lem::blockedStar}. 

Next suppose that $\mc{F}$ contains exactly one of $\{1,3\}$ and $\{1,4\}$; by symmetry we can assume $\{1,3\} \notin \mc{F}$ and $\{1,4\} \in \mc{F}$.  By (iii) there is $C \in \mc{F}$ such that $C \cap \{2,3,4\} = \{3\}$. Since $C \cap \{1,2\} \neq \es$ we have $1 \in C$. Since $\{1,3\} \notin \mc{F}$, without loss of generality $C = \{1,3,5\}$. Similarly, there is $D \in \mc{F}$ such that $C \cap D = \{5\}$. Since $D \cap \{1,2\} \neq \es$ and $D \cap \{1,4\} \neq \es$ we have $D = \{2,4,5\}$. Now no further sets not containing an existing set can be added to obtain an intersecting family, so $\mc{F}=R_5$ and $\lL(G) \leq \lL(K_5^3) - 10^{-3}$ by Lemma~\ref{lem::5VertexGenFam}.

Finally, suppose that $\{1,3\} \notin \mc{F}$ and $\{1,4\} \notin \mc{F}$. As in the previous case, without loss of generality $\{1,3,5\} \in \mc{F}$. Similarly, $\{1,4,i\} \in \mc{F}$ for some $i \geq 5$. If $\{2,5\} \in \mc{F}$, then $\{1,4,5\} \in \mc{F}$ and thus $\mc{F}$ is isomorphic to $R_5$, which was already considered. Also, since $G \ne K^3_5$ is intersecting, we can assume $n \ge 7$ by Lemma~\ref{lem::AtMost5Vertices} and Lemma~\ref{lem::AtMost6Vertices}. By (iv) there is $F_{67} \in G$ such that $\{6,7\} \sub F_{67}$. Since $\mc{F}$ is intersecting, $G$ is intersecting as well, so $F_{67} \cap \{2,3,4\} \neq \es$ and $F_{67} \cap \{1,2,5\} \neq \es$. It follows that $F_{67} = \{2,6,7\}$, but then $F_{67} \cap \{1,3,5\} = \es$, contradiction.
\end{proof}

We conclude this section by deducing Theorem~\ref{th::main3int} from the lemmas.

\medskip

\textbf{Proof of Theorem~\ref{th::main3int}}\ 
Let $\mc{F}$ be an intersecting $3$-graph on $[n]$. By Observation~\ref{obs::simple} and Lemma~\ref{lem::maxCoverPairs} we can assume that $\mc{F}$ is maximal and covers pairs. Then $Gen(n, 3, S(\mc{F})) = \mc{F}$ by Lemma~\ref{lem::genShift}. If $S(\mc{F})$ contains a set of size $1$ then $\mc{F} = Gen(n,3,pt)$, so $\lL(\mc{F}) < 2/27$ by Lemma~\ref{lem::star}. Assume then that $|F| \geq 2$ for every $F \in S(\mc{F})$. If $S(\mc{F})$ is $2$-uniform, then $S(\mc{F}) = K_3$ by Lemma~\ref{lem::uniqueIntersection}. Then $\mc{F} = Gen(n,3,K_3)$, so $\lL(\mc{F}) = 1/16$ by Lemma~\ref{lem::triangle}. If $S(\mc{F})$ is $3$-uniform, then $S(\mc{F}) = \mc{F}$ and thus $S(\mc{F})$ covers pairs. Let $k = \left|\bigcup_{F \in \mc{F}} F \right|$; then $k \le 7$ by Lemma~\ref{lem::uniqueIntersection} and Lemma~\ref{lem::genFamilyIsSmall}. If $k=7$ then $\mc{F}=F_7$ by Lemma~\ref{lem::F7IsUnique}, so $\lL(\mc{F})=1/27$ by Lemma~\ref{lem::F7}.
Otherwise, $k \leq 6$, so $\lL(\mc{F}) \leq \lL(K_5^3) - 10^{-3}$ by Lemma~\ref{lem::AtMost5Vertices} and Lemma~\ref{lem::AtMost6Vertices}. The only remaining case is that covered by Lemma~\ref{lem::Size2And3}. \qed

\section{An application to a hypergraph Tur\'an problem}
\label{sec::3graphTuran}

In this section we apply Theorem~\ref{th::main3int} to prove our main theorem on the Tur\'an number of $\mc{K}_{3,3}^3$, namely that for large $n$, the unique extremal example is $T_5^3(n)$, i.e.\ the balanced blow-up of $K^3_5$. First we note some simple facts about $T_5^3(n)$. It is $\mc{K}_{3,3}^3$-free, as for any attempted embedding of $\mc{K}_{3,3}^3$ in $T_5^3(n)$, there must be some $1 \le i, j \le 3$ such that $x_i$ and $y_j$ lie in the same part, but then $x_i y_j z_{ij}$ cannot be an edge. The number of edges satisfies
\[ t^3_5(n) = \sum_{0 \leq i < j < k \leq 4} \bfl{ \frac{n+i}{5} } \cdot \bfl{ \frac{n+j}{5} } \cdot \bfl{ \frac{n+k}{5} } = \frac{2}{25} n^3 + O(n^2). \]
Also, the minimum degree satisfies
\[ \dD^3_5(n) = t^3_5(n) - t^3_5(n-1) = \sum_{0 \leq i < j  \leq 3} \bfl{ \frac{n+i}{5} } \cdot \bfl{ \frac{n+j}{5} } = \frac{6}{25} n^2 + O(n).\]
We start by showing that the asymptotic result follows quickly from Theorem~\ref{th::main3int}. First we need some definitions. Suppose $F$ and $G$ are $r$-graphs. The \emph{Tur\'an density} of $F$ is $\pi(F) = \lim_{n\to\infty} \binom{n}{r}^{-1} \ex(n,F)$. Given $r$-graphs $F$ and $G$ we say $f:V(F) \to V(G)$ is a homomorphism if it preserves edges, i.e.\ $f(e) \in E(G)$ for all $e \in E(F)$. We say that $G$ is \emph{$F$-hom-free} if there is no homomorphism from $F$ to $G$. The \emph{blow-up density} is $b(G) = r! \lL(G)$. We say $G$ is \emph{dense} if every proper subgraph $G'$ satisfies $b(G') < b(G)$. We also need the following two standard facts (see e.g.~\cite[Section 3]{K}):
\begin{enumerate}[(i)]
\item $\pi(F)$ is the supremum of $b(G)$ over $F$-hom-free dense $G$,
\item dense $r$-graphs cover pairs.
\end{enumerate}

\begin{theorem} \label{th::AsymptoticTuran}
$\ex(n, \mc{K}_{3,3}^3) = \frac{2}{25} n^3 + o(n^3)$.
\end{theorem}

\begin{proof}
An equivalent formulation is that $\pi(\mc{K}_{3,3}^3) = \frac{12}{25}$. The lower bound is given by the construction $T_5^3(n)$. For the upper bound, by fact (i) above, it suffices to show that $b(G) \le 12/25$ for any $\mc{K}_{3,3}^3$-hom-free dense $G$. Suppose for a contradiction that $G$ is $\mc{K}_{3,3}^3$-hom-free, dense, and has $\lL(G) > 2/25$. By Theorem~\ref{th::main3int}, $G$ is not intersecting, so we can choose disjoint edges $\{x_1, x_2, x_3\}$ and $\{y_1, y_2, y_3\}$. Then by fact (ii), $G$ covers pairs, so for every $1 \leq i,j \leq 3$ there exists an edge $\{x_i, y_j, z_{ij}\}$. However, this defines a homomorphism from $\mc{K}_{3,3}^3$ to $G$, which contradicts $G$ being $\mc{K}_{3,3}^3$-hom-free.
\end{proof}

\subsection{Stability}
\label{subsec::stability}

In order to prove Theorem~\ref{th::main}, we will first prove the following stability result.

\begin{theorem} \label{prop::stability}
For any $\varepsilon > 0$ there exist $\dD > 0$ and an integer $n_0$ such that if $\mc{F}$ is a $\mc{K}_{3,3}^3$-free $3$-graph with $n \geq n_0$ vertices and at least $\left(\frac{2}{25} - \dD \right) n^3$ edges, then there exists a partition $V(\mc{F}) = A_1 \cup \ldots \cup A_5$ of the vertex set of $\mc{F}$ such that $\sum_{1 \leq i < j \leq 5} e(A_i \cup A_j) < \varepsilon n^3$.
\end{theorem}

We first show that it suffices to prove the result under the assumption that $\mc{F}$ is $\mc{K}_{3,3}^3$-hom-free. We need the Hypergraph Removal Lemma of R\"odl and Skokan \cite[Theorem 1.3]{RS}, which is as follows.

\begin{lemma}\label{remove}
Given $h \ge r \ge 2$, an $r$-graph $H$ on $h$ vertices, and $\aA>0$, there is $\bB>0$, such that, for any $r$-graph $G$ on $n$ vertices with at most $\bB n^h$ copies of $H$, one can delete $\aA n^r$ edges of $G$ to make it $H$-free.
\end{lemma}

We also need the following lemma which is a consequence of the result of Erd\H{o}s \cite{E} (see \cite[Section 2]{K}).

\begin{lemma}\label{blow-up}
Given an $r$-graph $H$ on $h$ vertices, $t \ge 1$ and $\bB>0$, there is an integer $n_0$ such that, if an $r$-graph $G$ on $n \ge n_0$ vertices has at least $\bB n^h$ copies of $H$ then $G$ contains the $t$-fold blow-up $H(t)$.
\end{lemma}

Suppose we have proved Theorem \ref{prop::stability} under the assumption that $\mc{F}$ is $\mc{K}_{3,3}^3$-hom-free. Let $\mc{H}$ be the set of homomorphic images of $\mc{K}_{3,3}^3$ up to isomorphism. Consider $H \in \mc{H}$ and let $t$ be such that $\mc{K}_{3,3}^3 \sub H(t)$. Let $\aA = \dD/|\mc{H}|$ and let $\bB$ be given by Lemma \ref{remove}. Since $\mc{F}$ is $\mc{K}_{3,3}^3$-free, it has at most  $\bB n^h$ copies of $H$ by Lemma \ref{blow-up}. Then by Lemma \ref{remove} we can delete $\aA n^3$ edges of $\mc{F}$ to make it $H$-free. Repeating this for all $H \in \mc{H}$, we can delete $\dD n^3$ edges of $\mc{F}$ to make it $\mc{K}_{3,3}^3$-hom-free. Then applying Theorem \ref{prop::stability} under this assumption with $\dD$ replaced by $2\dD$ we obtain the full theorem.

Henceforth we assume that $\mc{F}$ is $\mc{K}_{3,3}^3$-hom-free. Clearly, we can also assume that $\eps$ is sufficiently small. Let $\aA$, $\bB$, $\gG$ and $\dD$ be real numbers satisfying $1 \gg \gG \gg \bB \gg \aA \gg \eps \gg \dD \gg n_0^{-1}$.

Part of our proof follows the main ideas of~\cite{P}. We gradually change $\mc{F}$ (as well as some other related structure) by iterating a process which is called \emph{Symmetrization}. This process consists of two parts: \emph{Cleaning} and \emph{Merging}. It terminates as soon as we can no longer clean or merge anything. We refer to the basic object with which we operate as a \emph{pointed partitioned $3$-graph}; by this we mean a triple $(\mc{G},\mc{P},U)$ where $\mc{G} = (V,E)$ is a $3$-graph, $\mc{P} = \{P_u : u \in V\}$ is a partition of $V$ such that $u \in P_u$ for every $u \in V$, and $U \sub V$ is a transversal of $\mc{P}$. Note that $P_{u'} = P_u$ for all $u' \in P_u$ but we count each part in $\mc{P}$ once, i.e.\ it is a set, not a multiset. The precise description of the two parts of the process is as follows:

\noindent \textbf{Cleaning:}\\ 
\textbf{Input}: A pointed partitioned $3$-graph $(\mc{G},\mc{P},U)$ on $n$ vertices. \\
\textbf{Output}: A pointed partitioned $3$-graph $(\mc{G'},\mc{P}',U')$ on $n' \le n$ vertices. \\
\textbf{Process}: If the minimum degree of $\mc{G}$ is at least $\left(\frac{6}{25} - \aA \right) n^2$, then stop and return $(\mc{G}', \mc{P}', U') = (\mc{G}, \mc{P}, U)$. Otherwise, let $u \in V$ be an arbitrary vertex such that $deg_\mc{G}(u) < \left(\frac{6}{25} - \aA \right) n^2$. If $P_u = \{u\}$, then apply Cleaning to $(\mc{G} \sm u, \mc{P} \sm \{P_u\}, U \sm \{u\})$. Otherwise, let $v$ be an arbitrary vertex of $P_u \sm U$. Apply Cleaning to $(\mc{G} \sm v, \left(\mc{P} \sm \{P_u\} \right) \cup \{P_u \sm \{v\}\}, U)$.

\noindent \textbf{Merging:}\\ 
\textbf{Input}: A pointed partitioned $3$-graph $(\mc{G},\mc{P},U)$.\\
\textbf{Output}:  A pointed partitioned $3$-graph $(\mc{G'},\mc{P}',U')$ on the same vertex set as $\mc{G}$.\\
\textbf{Process}: If for every $u,v \in U$ there exists an edge $e \in E(\mc{G})$ such that $e \cap P_u \neq \es$ and $e \cap P_v \neq \es$, then return $(\mc{G}', \mc{P}', U') = (\mc{G}, \mc{P}, U)$. Otherwise, let $u,v \in U$ be two arbitrary vertices such that $deg_{\mc{G}}(u) \geq deg_{\mc{G}}(v)$ and $e \cap P_u = \es$ or $e \cap P_v = \es$ holds for every $e \in E(\mc{G})$. Merge $P_v$ into $P_u$, that is, for every $w \in V(\mc{G})$ let $P'_w = P_u \cup P_v$ if $w \in P_u \cup P_v$ and $P'_w = P_w$ otherwise. Moreover, let $U' = U \sm \{v\}$ and let $\mc{G}'$ be a blow-up of $\mc{G}[U']$ with vertex set $V(\mc{G})$, that is, the $3$-graph obtained from $\mc{G}[U']$ by replacing every vertex $u \in U'$ with the set of vertices $P'_u$ and every edge $\{u,v,w\} \in E(\mc{G}[U'])$ with the set of edges $\{\{a,b,c\} : a \in P'_u, b \in P'_v, c \in P'_w\}$. Let $\mc{P}' = \{P'_u : u \in V(\mc{G})\}$ and return $(\mc{G}', \mc{P}', U')$.

We are now ready to describe the entire symmetrization process:

\noindent \textbf{Symmetrization:}\\ 
\textbf{Input}: A $3$-graph $\mc{F} = (V,E)$.\\
\textbf{Output}: A pointed partitioned $3$-graph $(\mc{F}_{sym},\mc{P},U)$.\\
\textbf{Process}: Let $i=1$, let $\mc{H}_0 = \mc{F}$, let $U_0 = V$ and let $\mc{P}_0 = \{P_{0,u} : u \in V\}$, where $P_{0,u} = \{u\}$ for every $u \in V$, be a partition of $V$ into singletons. Let ($\mc{H}'_i = (V'_i, E'_i), \mc{P}'_i = \{P'_{i,u} : u \in V'_i\}, U'_i$) be the output of Cleaning($\mc{H}_{i-1}, \mc{P}_{i-1}, U_{i-1}$) and let $(\mc{H}_i = (V_i, E_i), \mc{P}_i = \{P_{i,u} : u \in V_i\}, U_i)$ be the output of Merging($\mc{H}'_i, \mc{P}'_i, U'_i$). If $(\mc{H}_i, \mc{P}_i, U_i) = (\mc{H}_{i-1}, \mc{P}_{i-1}, U_{i-1})$, then stop and return $(\mc{F}_{sym}, \mc{P}, U) = (\mc{H}_i, \mc{P}_i, U_i)$. Otherwise, increase $i$ by one and repeat Cleaning and Merging.

\medskip

Let $(\mc{F}_{sym}, \mc{P}, U)$ be the result of applying Symmetrization to $\mc{F}$. Let 
\[\mc{H}_0 = \mc{F}, \mc{H}'_1, \mc{H}_1, \ldots, \mc{H}'_{\ell}, \mc{H}_{\ell} = \mc{F}_{sym}\] 
be the sequence of $3$-graphs produced during this process, where $\mc{H}'_i = (V'_i, E'_i)$ and $\mc{H}_i = (V_i, E_i)$ for every $1 \leq i \leq \ell$. We split the proof into two stages. In the first stage we show that $\mc{F}_{sym}$ is a large and fairly balanced blow-up of $K_5^3$. In the second stage we show that $\mc{F}[V_\ell]$ is a subgraph of a blow-up of $K_5^3$; Theorem~\ref{prop::stability} will follow easily from this.

We start with the first stage, which we prove in a series of lemmas. The first four exhibit useful properties of the Symmetrization process whereas the last three deal with the resulting pointed partitioned $3$-graph $(\mc{F}_{sym},\mc{P},U)$. 

\begin{lemma} \label{lem::5properties}
The following properties hold for every $0 \leq i \leq \ell$.
\begin{description}
\item [$(P1)$] For every $0 \leq j \leq i$ the set $U_j \cap V_i$ is a transversal for the partition $\{P_{j,u} \cap V_i : u \in V_i\}$, where $V_0 = V(\mc{F}) = [n]$. In particular, $U_i$ is a transversal of $\mc{P}_i$.
\item [$(P2)$] $\mc{H}_i[U_i] = \mc{F}[U_i]$.
\item [$(P3)$] $|e \cap P_{i,u}| \leq 1$ for every $e \in E_i$ and every $u \in V_i$.
\item [$(P4)$] For every $u,v \in V_i$, if $v \in P_{i,u}$ then $\{e \sm \{v\} : e \in E_i\} = \{e \sm \{u\} : e \in E_i\}$.
\item [$(P5)$] If $i \geq 1$, then $U_i \sub U_{i-1}$ and $V_i \sub V_{i-1}$.
\end{description}
\end{lemma}

\begin{proof}
Clearly $(P1) - (P5)$ hold for $i=0$. It thus suffices to prove that they are maintained by an application of Cleaning followed by an application of Merging. Assume then that $(P1) - (P5)$ hold for some $i \geq 0$; we will prove that they hold for $i+1$ as well. For $(P3) - (P5)$ this follows directly from the definitions of Cleaning and of Merging. Property $(P2)$ holds for $i+1$ since $\mc{H}_{i+1}[U_{i+1}] = \mc{H}'_{i+1}[U_{i+1}] = \mc{H}_i[U_{i+1}] = \mc{F}[U_{i+1}]$, where the third equality follows since $\mc{H}_i[U_i] = \mc{F}[U_i]$ by assumption and $U_{i+1} \sub U_i$ by $(P5)$. It remains to show that $(P1)$ holds for $i+1$. Fix some $0 \leq j \leq i+1$. By $(P1)$ for $i$, if $0 \leq j \leq i$, then $U_j \cap V_i$ is a transversal of the partition $\{P_{j,u} \cap V_i : u \in V_i\}$. If $w \in \left(U_j \cap V_i \right) \sm \left(U_j \cap V_{i+1} \right)$ then since $w$ is the last vertex of $P_{j,w}$ to be deleted in Cleaning we have $P_{j,w} \cap V_{i+1} = \es$. It follows that $U_j \cap V_{i+1}$ is a transversal for the partition $\{P_{j,u} \cap V_{i+1} : u \in V_{i+1}\}$. On the other hand, if $j = i+1$, we note that each of Cleaning and Merging is defined to output a pointed partitioned $3$-graph, so $U_{i+1}$ is a transversal for the partition $\{P_{i+1,u} : u \in V_{i+1}\}$, as required.
\end{proof}

\bigskip

The property of Merging presented in the next lemma will only be used in the second stage of the proof.

\begin{lemma} \label{lem::propertOfMerging}
Let $1 \leq i \leq \ell$ and suppose that $P'_{i,v}$ was merged into $P'_{i,u}$ during the $i$th Merging step. 
Then $P'_{i,u} \cap V(\mc{F}_{sym}) = \es$ implies $P'_{i,v} \cap V(\mc{F}_{sym}) = \es$.
\end{lemma}

\begin{proof}
Assume that $P'_{i,v} \cap V(\mc{F}_{sym}) \neq \es$. Since $P'_{i,v} \cap V(\mc{F}_{sym}) \sub P_{i,u} \cap V(\mc{F}_{sym})$ by the definition of Merging, $P_{i,u} \cap V(\mc{F}_{sym}) \neq \es$. Since $U_i \cap P_{i,u} = \{u\}$, by Lemma~\ref{lem::5properties} $(P1)$ we have $u \in V(\mc{F}_{sym})$. Hence $u \in P'_{i,u} \cap V(\mc{F}_{sym})$, so $P'_{i,u} \cap V(\mc{F}_{sym}) \neq \es$ as claimed.
\end{proof}

\begin{lemma} \label{lem::K333free}
$\mc{H}_i$ is $\mc{K}_{3,3}^3$-hom-free for every $0 \leq i \leq \ell$.
\end{lemma}

\begin{proof}
This is true for $i=0$. It thus suffices to prove that no homomorphic copy of a $\mc{K}_{3,3}^3$ is created as a result of Cleaning and Merging. This is obvious for Cleaning. It holds for Merging by its definition and Lemma~\ref{lem::5properties} $(P3)$ and $(P4)$.
\end{proof}

\bigskip

The next lemma asserts that Merging does not decrease size.

\begin{lemma} \label{lem::NoDecreaseMerge}
$e(\mc{H}_i) \geq e(\mc{H}'_i)$ holds for every $1 \leq i \leq \ell$.
\end{lemma}

\begin{proof}
Fix some $1 \leq i \leq \ell$. Assume that $P'_{i,v}$ was merged into $P'_{i,u}$ during the $i$th merging step. By definition of Merging we have $deg_{\mc{H}'_i}(u) \geq deg_{\mc{H}'_i}(v)$, and $e \cap P'_{i,u} = \es$ or $e \cap P'_{i,v} = \es$ for every $e \in E'_i$. Hence, it follows by Lemma~\ref{lem::5properties} $(P4)$ that $e(\mc{H}_i) = e(\mc{H}'_i) + \left( deg_{\mc{H}'_i}(u) - deg_{\mc{H}'_i}(v) \right) |P'_{i,v}| \geq e(\mc{H}'_i)$.
\end{proof}

\bigskip

The next lemma asserts that the Symmetrization process does not require deleting too many vertices.

\begin{lemma} \label{lem::LargeFinalGraph}
$|V(\mc{F}_{sym})| \geq \left(1 - 2 \sqrt[3]{\dD/\aA} \right) n$.
\end{lemma}

\begin{proof}
Write $n_1 = |V(\mc{F}_{sym})|$. By Lemma~\ref{lem::NoDecreaseMerge} and the definition of Cleaning we have
\[ e(\mc{F}_{sym}) \ge e(\mc{F}) - \sum_{j=n_1}^n (\tfrac{6}{25}-\aA)j^2 
=  e(\mc{F}) - (\tfrac{2}{25}-\tfrac{\aA}{3})(n^3-n_1^3) + O(n^2).\]
Assume first that $n_1 \geq n_0/2$. Since $\mc{F}_{sym}$ is $\mc{K}_{3,3}^3$-free by Lemma~\ref{lem::K333free} and $n_0^{-1} \ll \dD$, by Theorem~\ref{th::AsymptoticTuran} we have $e(\mc{F}_{sym}) \leq (\tfrac{2}{25} + \dD) n_1^3$. This gives
\[ e(\mc{F}) \le (\tfrac{2}{25}-\tfrac{\aA}{3})n^3 + (\tfrac{\aA}{3}+\dD)n_1^3 + O(n^2).\]
Recalling that $e(\mc{F}) \geq (\tfrac{2}{25} - \dD) n^3$, we estimate $\tfrac{\aA}{4}(n^3-n_1^3) \le \dD(n^3+n_1^3)$. This implies 
$\tfrac{\aA}{4}(n-n_1)^3 \le 2\dD n^3$, so $n-n_1 \le 2 \sqrt[3]{\dD/\aA} \cdot n$, as claimed. On the other hand, if $n_1 < n_0/2$ then $n-n_1 > n/2 > 2 \sqrt[3]{\dD/\aA} \cdot n$. But applying the same calculation as above to the $3$-graph obtained from $\mc{F}$ by deleting the first $\lceil 2 \sqrt[3]{\dD/\aA} \cdot n \rceil + 1$ vertices results in a contradiction.
\end{proof}

\begin{lemma} \label{lem::FsymIntersecting}
$\mc{F}[U_{\ell}]$ is intersecting.
\end{lemma}

\begin{proof}
By the stopping rule for Symmetrization, $\mc{F}_{sym}[U_{\ell}]$ covers pairs. By Lemma~\ref{lem::5properties} $(P2)$, the same holds for $\mc{F}[U_{\ell}]$. Assume for the sake of contradiction that there exist edges $\{a_1, a_2, a_3\}$ and $\{b_1, b_2, b_3\}$ of $\mc{F}[U_{\ell}]$ such that $\{a_1, a_2, a_3\} \cap \{b_1, b_2, b_3\} = \es$. Since $\mc{F}[U_{\ell}]$ covers pairs, for every $1 \leq i,j \leq 3$ there is $x_{ij} \in V(\mc{F}[U_{\ell}])$ such that $\{a_i, b_j, x_{ij}\} \in E(\mc{F}[U_{\ell}])$. Hence, $\mc{F}[U_{\ell}]$ is not $\mc{K}_{3,3}^3$-hom-free, so $\mc{F}_{sym}$ is not $\mc{K}_{3,3}^3$-hom-free, contrary to Lemma~\ref{lem::K333free}.  
\end{proof}

\begin{lemma} \label{lem::FinalGraphIsK53}
$\mc{F}_{sym}[U_{\ell}] \cong K_5^3$ and $(1/5 - \bB) n \leq |A| \leq (1/5 + \bB) n$ for every part $A \in \mc{P}$. 
In particular $|\mc{P}| = 5$.
\end{lemma}

\begin{proof}
For every $u \in [n]$ let $y_u = |P_{\ell, u}|/|V_{\ell}|$ if $u \in U_{\ell}$ and $y_u = 0$ otherwise. Let $y = (y_1, \ldots, y_n)$; note that $\sum_{u \in [n]} y_u = 1$. By Lemma~\ref{lem::LargeFinalGraph} we have
\[\lL(\mc{F}[U_{\ell}]) \geq p_{\mc{F}[U_{\ell}]}(y) \geq \frac{|E_{\ell}|}{|V_{\ell}|^3} \geq \frac{\left(\frac{2}{25} - \dD \right) n^3 - 2 \sqrt[3]{\dD/\aA} \cdot n \cdot \binom{n-1}{2}}{n^3} \geq \frac{2}{25} - \dD - \sqrt[3]{\dD/\aA}.\]

Since $\dD \ll \aA \ll 1$, by Lemma~\ref{lem::FsymIntersecting} and Theorem~\ref{th::main3int} we have $\mc{F}[U_{\ell}] \cong K_5^3$. Hence, $\mc{F}_{sym}[U_{\ell}] \cong K_5^3$ holds by Lemma~\ref{lem::5properties} $(P2)$. Assume for the sake of contradiction that there is $A \in \mc{P}$ such that $||A| - n/5| > \bB n$. Then there must be some $B \in \mc{P}$ such that $|B| < (n/5 - \bB n/4)$. 
It follows that 
\[e(\mc{F}_{sym}) \leq 4 (n/5 + \bB n/16)^3 + 6 (n/5 - \bB n/4) (n/5 + \bB n/16)^2 \leq 2n^3/25 - \bB^2 n^3/40.\] 
Then Lemma~\ref{lem::LargeFinalGraph} gives $e(\mc{F}) \leq e(\mc{F}_{sym}) + \sqrt[3]{\dD/\aA} \cdot n^3 
< (\tfrac{2}{25} - \dD)n^3$, a contradiction.     
\end{proof}

This completes the first stage of the proof. The second stage is to show that $\mc{F}[V_{\ell}]$ is a subgraph of a blow-up of $K_5^3$. To do so, we will reverse the Merging steps performed during Symmetrization (this process was called Splitting in~\cite{P}). By Lemma~\ref{lem::FinalGraphIsK53} we have $|U_{\ell}| = 5$, so we can identify $U_\ell$ with $[5]$. For every $0 \leq i \leq \ell$ we will find a partition $\mc{Q}_i = \{Q_{i,j} : 1 \leq j \leq 5\}$ of $V_{\ell}$ which satisfies the following three properties:
\begin{description} 
\item [$(R1)$] $j \in Q_{i,j}$ holds for every $1 \leq j \leq 5$,
\item [$(R2)$] for every $v \in V_{\ell}$ there exists some $1 \leq u \leq 5$ such that $P_{i,v} \cap V_{\ell} \sub Q_{i,u}$, that is, the restriction of  $\mc{P}_i$ to $V_\ell$ is a refinement of $\mc{Q}_i$,
\item [$(R3)$] $\mc{H}_i[Q_{i,p} \cup Q_{i,q}] = \es$ for $1 \le p,q \le 5$, that is, $\mc{H}_i[V_{\ell}]$ is a subgraph of a blow-up of $K_5^3$ with parts $Q_{i,1}, \ldots, Q_{i,5}$.
\end{description}

Set $\mc{Q}_{\ell} = \mc{P}_{\ell}$. It follows by Lemma~\ref{lem::FinalGraphIsK53} that $\mc{Q}_{\ell}$ satisfies $(R1) - (R3)$. Assume that for some $1 \leq i \leq \ell$ we have already defined a partition $\mc{Q}_i$ which satisfies $(R1) - (R3)$; we will show how to define a partition $\mc{Q}_{i-1}$ with the desired properties. 

We adopt the following notation. Let $u, v \in U'_i$ be such that in the $i$th Merging step $P'_{i,v}$ was merged into $P'_{i,u}$. Note that $u \in U_i$ but $v \notin U_i$. Write 
\[ \mc{G}_i = \mc{H}_i[V_{\ell}], \quad \mc{G}_{i-1} = \mc{H}_{i-1}[V_{\ell}], 
\quad A_u = P'_{i,u} \cap V_{\ell} \quad \text{ and } \quad A_v = P'_{i,v} \cap V_{\ell}.\]
We can view  $\mc{G}_i$ as being obtained from $\mc{G}_{i-1}$ by Merging $A_v$ into $A_u$.
Since $\mc{Q}_i$ satisfies $(R2)$, we can assume without loss of generality that \[A_u \cup A_v \sub Q_{i,1}.\] 
Write \[ W_1 = Q_{i,1} \sm A_v, \quad  W_j = Q_{i,j} \quad \text{ for } \quad 2 \leq j \leq 5, 
\quad \text{ and } \quad W = V_\ell \sm A_v = \cup_{j=1}^5 W_j. \]
We can assume $A_v \ne \es$, otherwise setting $\mc{Q}_{i-1} = \mc{Q}_i$ yields the desired partition. 
Then $A_u \neq \es$ by Lemma~\ref{lem::propertOfMerging}, so $W_1 \neq \es$. 
Moreover, $A_u \neq \es$, $A_v \neq \es$ and Lemma~\ref{lem::5properties} $(P1)$ imply $u, v \in V_{\ell}$. 
We also note that any triple $e \sub W$ belongs to $E_i$ if and only if it belongs to $E_{i-1}$.

During the next series of lemmas, we will show that there exists some $1 \leq j \leq 5$ such that adding $A_v$ to $W_j$ yields the desired partition $\mc{Q}_{i-1}$. Write \[ m = |V_{\ell}|\]
and let $\mc{B}_i$ be a blow-up of $K_5^3$ with parts $\{Q_{i,j} : 1 \leq j \leq 5\}$.

\begin{lemma} \label{lem::degrees} For every $0 \leq i \leq \ell$ and $1 \leq j \leq 5$,
\begin{enumerate}[(i)]
\item $\dD(\mc{G}_i) \geq \left(\tfrac{6}{25} - 2\aA \right) m^2$, so $e(\mc{G}_i) \geq \left(\tfrac{2}{25} - \tfrac{2\aA}{3} \right) m^3$, 
\item $||Q_{i,j}| - m/5| \leq \gG m/100$, 
\item $d_{\mc{B}_i \sm \mc{G}_i}(u) \leq \gG m^2$ for every $u \in V_{\ell}$.
\end{enumerate}
\end{lemma}

\begin{proof}
Fix some $0 \leq i \leq \ell$. By definition of Cleaning, $d_{\mc{H}_i}(v) \ge (\tfrac{6}{25} - \aA) m^2$ holds for every $v \in V_{\ell}$. It follows by Lemma~\ref{lem::LargeFinalGraph} that
\begin{align*}
\dD(\mc{G}_i) \ge (\tfrac{6}{25} - \aA) m^2 - |V_i \sm V_{\ell}| n \geq \left(\tfrac{6}{25} - \aA \right) m^2 - 2 \sqrt[3]{\dD/\aA} \cdot n^2 \geq \left(\tfrac{6}{25} - 2\aA \right) m^2.
\end{align*}
This proves (i). Next, assume for the sake of contradiction that there exists some $1 \leq j \leq 5$ such that $||Q_{i,j}| - m/5| > \gG m/100$. Then there must exist some $1 \leq p \leq 5$ such that $|Q_{i,p}| < m/5 - \gG m/400$. Let $1 \leq q \leq 5$ be such that $|Q_{i,q}| = \max \{|Q_{i,a}| : 1 \leq a \leq 5\}$, and let $x \in Q_{i,q}$ be an arbitrary vertex. Since $\mc{Q}_i$ satisfies $(R3)$, it follows that
\begin{align*}
\dD(\mc{G}_i) & \leq d_{\mc{G}_i}(x) \leq \tbinom{3}{2} (m/5 + \gG m/1600)^2 + \tbinom{3}{1} (m/5 - \gG m/400) (m/5 + \gG m/1600)\\ 
& < \left(\tfrac{6}{25} - 2\aA \right) m^2
\end{align*}
contrary to (i). This proves (ii). Finally, let $x \in V_{\ell}$ be an arbitrary vertex. By (ii) we have
\[d_{\mc{B}_i}(u) \leq \tbinom{4}{2} (m/5 + \gG m/400)^2 \leq \left(\tfrac{6}{25} + \gG/2 \right) m^2.\]
Since $\mc{G}_i \sub \mc{B}_i$ by $(R3)$ for $\mc{Q}_i$, this implies (iii), using (i) and $\aA \ll \gG$.
\end{proof}    

Next we need some more notation and terminology. Suppose the partition $\{W'_i: 1 \le i \le 5\}$ of $V_\ell$ is obtained from $\{W_i: 1 \le i \le 5\}$ by adding $A_v$ to some part. We call an edge $e \in E_{i-1}$ \emph{bad} if $|e \cap W'_p| = 2$ for some $1 \leq p \leq 5$, \emph{very bad} if $|e \cap W'_p| = 3$ for some $1 \leq p \leq 5$, or \emph{good} otherwise. Suppose without loss of generality that adding $A_v$ to $W_5$ minimises
\[ \Ss := \sum_{1 \leq p < q \leq 5} e_{\mc{G}_{i-1}}(W'_p \cup W'_q) - 2\sum_{p=1}^5 e_{\mc{G}_{i-1}}(W'_p).\]
For every $1 \leq j \leq 5$ let $Q_{i-1,j} = W'_j$ (where $\{W'_i: 1 \le i \le 5\}$ is obtained from $\{W_i: 1 \le i \le 5\}$ by adding $A_v$ to $W_5$) and let $\mc{Q}_{i-1} := Q_{i-1,1} \cup \ldots \cup Q_{i-1,5}$. We will prove that $\mc{Q}_{i-1}$ satisfies $(R1) - (R3)$. This is immediate for $(R2)$, and $(R1)$ follows since $W_1 \neq \es$ and $A_v \cap [5] = \es$ by definition of Merging. It remains to show $(R3)$, i.e.\ that all edges are good. Equivalently, we need show that $\Ss=0$, as every bad edge is counted exactly once in $\Ss$, and every very bad edge is counted exactly twice in $\Ss$, whereas good edges are not counted at all.

First we note that any $e \in E_{i-1}$ that is not good satisfies
\[ |e \cap A_v| = 1.\]
This holds as $|e \cap A_v| \leq 1$ by Lemma~\ref{lem::5properties} $(P3)$ and $|e \cap A_v| \geq 1$ by $(R3)$ for $\mc{Q}_i$.

We say that a vertex of $A_v$ is \emph{bad} if it is contained in at least $10^{-3} m^2$ edges which are not good. Before proving that all edges are good, we will prove that a vertex of $A_v$ cannot be contained in too many edges which are bad or very bad. 

\begin{lemma} \label{lem::noBadVertex}
There are no bad vertices.
\end{lemma}

\begin{proof}
Assume for the sake of contradiction that $x \in A_v$ is a bad vertex. We consider two cases according to the number of edges of $E_{i-1}$ containing $x$ and a vertex of $W_5$. Suppose first that there are at least $10^{-3} m^2/2$ such edges. It follows that there exists an index $1 \leq p \leq 5$ for which there are at least $10^{-4} m^2$ edges $\{x,y,z\} \in E_{i-1}$ such that $y \in W_5$ and $z \in W_p$. Fix such a $p$ and fix $\{a,b,c\} \sub \{1,2,3,4\} \sm \{p\}$. For every $q \in \{a,b,c\}$ let 
\[B_q(x) := \{w \in W_q : \exists w' \in V_{\ell} \textrm{ such that } \{x, w, w'\} \in E_{i-1}\}.\]
Consider the subcase that $|B_q(x)| \geq 10^{-5} m$ holds for every $q \in \{a,b,c\}$. Consider a maximal matching $\{\{x^a_j, x^b_j, x^c_j\} : j \in J \}$ in $E_i$ such that $x^q_j \in B_q(x)$ for all $q \in \{a,b,c\}$, $j \in J$. We claim that $|J| \ge \tfrac{1}{2} \cdot 10^{-5} m$. For otherwise we have $B'_q \sub B_q(x)$, $q \in \{a,b,c\}$ of size $\tfrac{1}{2} \cdot 10^{-5} m$ such that there is no $\{x^a,x^b,x^c\} \in E_i$ with $x^q \in B_q(x)$ for all $q \in \{a,b,c\}$. But every such triple is in $\mc{B}_i \sm \mc{G}_i$, so this contradicts Lemma~\ref{lem::degrees}(iii). Thus $|J| \ge \tfrac{1}{2} \cdot 10^{-5} m$. Also, each edge $\{x^a_j, x^b_j, x^c_j\}$, $j \in J$ belongs to $E_{i-1}$, as it is disjoint from $A_v$.

Fix an arbitrary edge $\{x,y,z\} \in E_{i-1}$ such that $y \in W_5$ and $z \in W_p$ and note that for every $j \in J$ there must be some $q \in \{a,b,c\}$ and $w \in \{y,z\}$ such that $\{x^q_j, w, w'\} \notin E_{i-1}$ for all $w' \in V_{\ell}$; otherwise, by definition of the $B_q(x)$'s, we would contradict $\mc{G}_{i-1}$ being $\mc{K}_{3,3}^3$-hom-free. Assume without loss of generality that $\{x^a_j, y, w'\} \notin E_{i-1}$ for all $w' \in V_{\ell}$ for at least $|J|/6 \ge \tfrac{1}{12} \cdot 10^{-5} m$ indices $j \in J$, and also $b \ne 1$ (so that $A_v \cap Q_{i,b} = \es$). By Lemma~\ref{lem::degrees}(ii) there are at least $m/5 - \gG m/100$ choices of $w' \in Q_{i,b}$, so we obtain at least $\tfrac{1}{12} \cdot 10^{-5} m \cdot (m/5 - \gG m/100) > \gG m^2$ triples of $\mc{B}_i \sm \mc{G}_i$ containing $y$, contrary to Lemma~\ref{lem::degrees}(iii). 

In the other subcase, we can assume without loss of generality that $|B_a(x)| < 10^{-5} m$. Now consider the partition $\{W''_i: 1 \le i \le 5\}$ of $V_\ell$ obtained by adding $A_v$ to $W_a$ rather than $W_5$. By Lemma~\ref{lem::5properties} $(P4)$ applied to $A_v$, under this new partition, every very bad edge $\{w, w', w''\}$ such that $w \in A_v$ and $w', w'' \in W_5$ becomes bad, and every bad edge $\{w, w', w''\}$ such that $w \in A_v$, $w' \in W_5$ and $w'' \in W_p$ becomes good. It follows by our assumptions on $x$ that there are at least $|A_v| \cdot 10^{-4} m^2$ such edges. Moreover, every good edge which turned bad, and every bad edge which turned very bad, must be of the form $\{w, w', w''\}$ where $w \in A_v$ and $w' \in W_a$. By  Lemma~\ref{lem::5properties} $(P4)$ there are at most $|A_v| |B_a(x)| m < |A_v| 10^{-5} m^2$ such edges. However, this contradicts the minimality of $\Ss$.

The second case is that there are less than $10^{-3} m^2/2$ edges of $E_{i-1}$ containing $x$ and a vertex of $W_5$. Let $t \in [4]$ be such that there is some $\{x, w_t, w'_t\} \in E_{i-1}$ with $w_t$ and $w'_t$ in $W_t$; such an index $t$ exists since $x$ is bad. Fix $\{a,b,c\} \sub [5] \sm \{1,t\}$ and let $\{\{x^a_j, x^b_j, x^c_j\} : j \in J_t \}$ be a maximal matching in $E_i$ such that $x^q_j \in W_q$ for all $q \in \{a,b,c\}$, $j \in J_t$. Similarly to the previous case, we have $|J_t| \ge  (1/5 - 2 \sqrt{\gG}) m$; otherwise, using Lemma~\ref{lem::degrees}(ii), we would contradict Lemma~\ref{lem::degrees}(iii). Also, each edge $\{x^a_j, x^b_j, x^c_j\}$, $j \in J_t$ belongs to $E_{i-1}$, as it is disjoint from $A_v$. Since $\mc{G}_{i-1}$ is $\mc{K}_{3,3}^3$-hom-free, for every $j \in J_t$, there must exist $d \in \{x, w_t, w'_t\}$ and $d' \in \{x^a_j, x^b_j, x^c_j\}$ such that $\{d, d', d''\} \notin E_{i-1}$ for any $d'' \in V_{\ell}$.

Consider the subcase that there are at least $60 \gG m$ indices $j \in J_t$ such that $(d, d') \in \{w_t, w'_t\} \times \{x^a_j, x^b_j, x^c_j\}$. Then without loss of generality $\{w_t, x^a_j, d''\} \notin E_{i-1}$ for every $d'' \in V_{\ell}$ and at least $10 \gG m$ indices $j \in J_t$, and also $b \ne 1$. However, this gives at least $10 \gG m \cdot |Q_{i,b}| > \gG m^2$ triples of $\mc{B}_i \sm \mc{G}_i$ containing $w_t$, contrary to Lemma~\ref{lem::degrees}(iii). 

In the remaining subcase there is $I_t \sub J_t$ such that $|I_t| \geq (1/5 - 3 \sqrt{\gG}) m$ and for every $j \in I_t$ there exists some $q \in \{a, b, c\}$ such that $\{x, x^q_j, w\} \notin E_{i-1}$ for every $w \in V_{\ell}$. Then the degree of $x$ in $\mc{G}_{i-1}$ is at most 
\[\tbinom{4}{2} (m/5 + \gG m/100)^2 + 4 \tbinom{m/5 + \gG m/100}{2} + 10^{-3} m^2/2 - |I_t| (m-2)/2
 < \left(\tfrac{6}{25} - 2\aA \right) m^2 \,,\]
contrary to Lemma~\ref{lem::degrees}(i). We conclude that there are no bad vertices.
\end{proof}

In our next lemma we will conclude the second stage of the proof by showing that every edge of $E_{i-1}$ is good. First we observe that  $|W_1| \geq m/6$, as otherwise, considering any $w \in A_v$, by Lemmas~\ref{lem::degrees}(ii) and~\ref{lem::noBadVertex} we have
\begin{align*}
d_{\mc{G}_{i-1}}(w) & \leq 10^{-3} m^2 + \sum_{1 \leq p < q \leq 4} |W_p||W_q|\\ 
& \leq 10^{-3} m^2 + \tbinom{3}{2} (m/5 + \gG m/100)^2 + \tbinom{3}{1} (m/5 + \gG m/100) m/6\\ 
& < \left(\tfrac{6}{25} - 2\aA \right) m^2 \,, 
\end{align*}
contrary to Lemma~\ref{lem::degrees}(i). We can now state our next lemma.

\begin{lemma}
Every edge of $E_{i-1}$ is good.
\end{lemma}

\begin{proof}
Assume for the sake of contradiction that $\{x, y, z\} \in E_{i-1}$ is not a good edge. Without loss of generality $x \in A_v$, $y \in W_4 \cup W_5$ and $z \in W_4 \cup W_5$. Consider a maximal matching $\{x^1_j, x^2_j, x^3_j\}$, $j \in J$ in $E_i$ such that $x^q_j \in W_q$ for all $q \in \{1,2,3\}$, $j \in J$. Using $|W_1| \ge m/6$ and Lemma~\ref{lem::degrees} (ii) we have $|J| \ge m/10$, otherwise we would contradict Lemma~\ref{lem::degrees}(iii). Also, each edge $\{x^1_j, x^2_j, x^3_j\}$, $j \in J$ belongs to $E_{i-1}$, as it is disjoint from $A_v$. 
Since $\mc{G}_{i-1}$ is $\mc{K}_{3,3}^3$-hom-free, for every $j \in J$ there exist some $1 \leq q \leq 3$ and $w \in \{x, y, z\}$ such that $\{x^q_j, w, w'\} \notin E_{i-1}$ for every $w' \in V_{\ell}$. 

Consider the case that there are at least $m/90$ indices $j \in J$ for which $\{x^1_j, x, w'\} \notin E_{i-1}$ for every $w' \in V_{\ell}$. Since $x$ is not bad, by Lemma~\ref{lem::degrees}(ii)
\[ d_{\mc{G}_{i-1}}(x) \le \tbinom{4}{2} (m/5 + \gG m/100)^2 - m^2/91 + 10^{-3} m^2 < \left(\tfrac{6}{25} - 2\aA \right) m^2,\]
contrary to Lemma~\ref{lem::degrees}(i). In the other case, without loss of generality there are at least $m/90$ indices $j \in J$ for which $\{x^1_j, y, w'\} \notin E_{i-1}$ for every $w' \in V_{\ell}$. There are at least $m/6$ choices of $w' \in W_2$, each giving a triple of $\mc{B}_i \sm \mc{G}_i$ containing $y$, so we obtain at least $m/90 \cdot m/6 > \gG m^2$ such triples, contrary to Lemma~\ref{lem::degrees}(iii). 
\end{proof}    

This shows that $\mc{Q}_{i-1}$ satisfies $(R3)$, so Splitting has the required properties. It terminates with $\mc{Q}_0$ such that $\mc{F}[V_{\ell}] = \mc{H}_0[V_{\ell}]$ is a subgraph of a blow-up of $K_5^3$ with parts $Q_{0,1}, \ldots, Q_{0,5}$. Let $A_1,\dots,A_5$ be obtained from ${\mathcal Q}_0$ by adding $V(\mc{F}) \sm V_{\ell}$ to $Q_{0,1}$. Then by Lemma~\ref{lem::LargeFinalGraph} we have $\sum_{1 \leq i \leq j \leq 5} e(A_i \cup A_j) \leq 2 \sqrt[3]{\dD/\aA} \cdot n^3 < \eps n^3$. This concludes the proof of Theorem \ref{prop::stability}.

\subsection{Proof of Theorem~\ref{th::main}}

First, we claim that it suffices to prove Theorem~\ref{th::main} under the additional assumption that the minimum degree of $\mc{F}$ is at least $\dD_5^3(n)$. Indeed, assume we have proved Theorem~\ref{th::main} for every maximum $\mc{K}_{3,3}^3$-free $3$-graph $\mc{F}$ with $n \geq n_0$ vertices and minimum degree at least $\dD_5^3(n)$. Let $\mc{H}_n$ be a maximum $\mc{K}_{3,3}^3$-free $3$-graph on $n \geq n_0^3$ vertices. If there exists a vertex $u_n \in V(\mc{H}_n)$ whose degree is strictly smaller than $\dD_5^3(n)$, delete it; that is, replace $\mc{H}_n$ with $\mc{H}_{n-1} := \mc{H}_n \sm \{u_n\}$. Repeating this process, where at each step we delete vertices whose degree is too small with respect to the order of the current hypergraph, we end up with a hypergraph $\mc{H}_m$ on $m$ vertices whose minimum degree is at least $\dD_5^3(m)$ (or with an empty hypergraph). We claim that $m \geq n_0$ (in particular, $\mc{H}_m \neq \es$). For every $m \leq i \leq n$ let $f(i) := e(\mc{H}_i) - t_5^3(i)$. Note that $f(i) \leq \binom{i}{3}$ holds for every $m \leq i \leq n$ and that $f(n) \geq 0$ holds by maximality. If, for some $m \leq i \leq n$ and for some $u_i \in V(\mc{H}_i)$, we have $d_{\mc{H}_i}(u_i) \leq \dD_5^3(i) - 1$, then $f(i-1) = e(\mc{H}_{i-1}) - t_5^3(i-1) = e(\mc{H}_i)- d_{\mc{H}_i}(u_i) - (t_5^3(i) - \dD_5^3(i)) \geq f(i) + 1$. It follows that $f(m) \geq f(n) + n - m \geq n - m$. Hence, $m \geq n - f(m) \geq n_0^3 - \binom{n_0}{3} \geq n_0$. If $m=n$, then there is nothing to prove. Otherwise, deleting $u_n, \ldots, u_{m+1}$ from the maximal $\mc{H}_n$ yields a $\mc{K}_{3,3}^3$-free $3$-graph $\mc{H}_m$ with $m \geq n_0$ vertices, minimum degree at least $\dD_5^3(m)$ and strictly more than $t_5^3(m)$ edges. This contradicts our assumption.

Now let $\aA, \bB, \gG, \dD$ and $\eps > 0$ be real numbers satisfying $\eps \ll \dD \ll \gG \ll \bB \ll \aA \ll 1$. Let $\mc{F} = (V,E)$ be a maximum size $\mc{K}_{3,3}^3$-free $3$-graph on $n$ vertices, where $n$ is sufficiently large. Clearly $|E| \geq t_5^3(n)$, and as noted above we can assume $\dD(\mc{F}) \ge \dD_5^3(n)$. Let $V = A_1 \cup \ldots \cup A_5$ be a partition of the vertex set of $\mc{F}$ which minimizes \[ \Ss := \sum_{1 \leq i < j \leq 5} e(A_i \cup A_j) - 2\sum_{i=1}^5 e(A_i).\] 
By Theorem~\ref{prop::stability} we can assume that $\Ss < \eps n^3$. Similarly to the proof of Lemma~\ref{lem::FinalGraphIsK53}, we have $||A_i| - n/5| \leq \dD n$ for every $1 \leq i \leq 5$; otherwise we would have the contradiction $|E| \le 2n^3/25 - \dD^2 n^3/40 + \eps n^3 <  t_5^3(n)$.

Similarly to the proof of Theorem~\ref{prop::stability}, we call an edge $e \in E$ \emph{bad} if there exists some $1 \leq i \leq 5$ such that $|e \cap A_i| = 2$, \emph{very bad} if there exists some $1 \leq i \leq 5$ such that $|e \cap A_i| = 3$, or \emph{good} otherwise. Note that every bad edge is counted exactly once in $\Ss$ and every very bad edge is counted exactly twice in $\Ss$ (good edges are not counted at all). We call a vertex $u \in V$ \emph{bad} if it is incident with at least $40 \aA n^2$ edges which are not good.

\begin{lemma}
There are no bad vertices.
\end{lemma}

\begin{proof}
Assume for the sake of contradiction that $u \in V$ is a bad vertex. Without loss of generality $u \in A_1$. We consider two cases according to whether there are at least $20 \aA n^2$ edges containing $u$ and another vertex in $A_1$. Suppose first that such edges exist, and without loss of generality there are at least $4 \aA n^2$ edges $e$ containing $u$ in $A_1 \cup A_2$ such that $|e \cap A_2| \leq 1$. Let $\{w_k,w'_k\}$, $k \in K$ be a maximal collection of pairwise disjoint pairs in $A_1 \cup A_2$ such that $\{u,w_k,w'_k\} \in E$ for all $k \in K$; then $|K| \ge 4\aA n$. 
For every $3 \leq i \leq 5$ let $B_i(u)$ be the set of $v \in A_i$ such that $\{u,v\}$ is contained in at least $15$ edges of $E$.

Consider the subcase that $|B_i(u)| \geq \bB n$ for every $3 \leq i \leq 5$. Consider a maximal matching $\{x^3_j,x^4_j,x^5_j\}$, $j \in J$ in $E$ such that $x^q_j \in B_q(u)$ for all $q \in \{3,4,5\}$, $j \in J$. Then $|J| \ge \bB n/2$, otherwise we obtain at least $(\bB n/2)^3 > \Ss$ triples in $(A_3 \times A_4 \times A_5) \sm E$, which contradicts $|E| \geq t_5^3(n)$. Since $\mc{F}$ is $\mc{K}_{3,3}^3$-free, for every $j \in J$ and $k \in K$ there are $a \in \{u,w_k,w'_k\}$ and $b \in \{x^3_j,x^4_j,x^5_j\}$ such that there are at most $14$ edges of $E$ containing $\{a,b\}$; otherwise we can greedily choose vertices to extend $\{u,w_k,w'_k\}$ and $\{x^3_j,x^4_j,x^5_j\}$ to a copy of $\mc{K}_{3,3}^3$. By definition of $B_i(u)$ we have $a \in \{w_k,w'_k\}$. However, this gives at least $|K| \cdot |J| \cdot n/2 > \Ss$ triples $e \notin E$ with $|e \cap A_i| \le 1$ for $1 \le i \le 5$, which contradicts $|E| \geq t_5^3(n)$. 

In the other subcase, we can assume without loss of generality that $|B_3(u)| < \bB n$. Consider the partition $\{A'_i: 1 \le i \le 5\}$ of $V$ obtained from $\{A_i: 1 \le i \le 5\}$ by moving $u$ from $A_1$ to $A_3$. In this new partition, every very bad edge $\{u, x, y\}$ such that $x,y \in A_1$ becomes bad, and every bad edge $\{u, x, y\}$ such that $x \in A_1$ and $y \in A_2$ becomes good. By assumption there are at least $4 \aA n^2$ such edges. Moreover, every good edge which turned bad, and every bad edge which turned very bad, must be of the form $\{u, v, w\}$ where $v \in A_3$. There are at most $n |B_3(u)| + 14n  < 4 \aA n^2$ such edges. However, this contradicts minimality of $\Ss$.

The second case is that there are less than $20 \aA n^2$ edges containing $u$ and another vertex in $A_1$. Let $2 \leq t \leq 5$ be such that there are at least $4 \aA n^2$ edges $\{u,v,w\} \in E$ with $v$ and $w$ in $A_t$; such an index $t$ exists since $u$ is bad. Let $\{a,b,c\} = \{2,3,4,5\} \sm \{t\}$. Let $\{v^t_k,w^t_k\}$, $k \in K_t$ be a maximal collection of pairwise disjoint pairs in $A_t$ such that $\{u,v^t_k,w^t_k\} \in E$ for all $k \in K_t$; then $|K_t| \ge 4\aA n$. Let $\{x^a_j, x^b_j, x^c_j\}$, $j \in J_t$ be a maximal matching in $E$ such that $x^q_j \in A_q$ for all $q \in \{a,b,c\}$, $j \in J_t$. Then $|J_t| \ge (1/5 - 2 \dD) n$, otherwise, using $|A_i| \ge (1/5 - \dD) n$ for $1 \le i \le 5$ and $\Ss < \eps n^3$, we would contradict $|E| \geq t_5^3(n)$. Since $\mc{F}$ is $\mc{K}_{3,3}^3$-free, for every $j \in J_t$ and $k \in K_t$ there are $x \in \{x^3_j,x^4_j,x^5_j\}$ and $y \in \{u,v^t_k,w^t_k\}$ such that $\{x,y\}$ is contained in at most $14$ edges.

Suppose first that there are at least $\dD n^2$ pairs $(j,k) \in J_t \times K_t$ such that $y \in \{v^t_k,w^t_k\}$. It follows that there are at least $\dD n^2 \cdot n/2 > \Ss$ triples $e \notin E$ with $|e \cap A_i| \le 1$ for $1 \le i \le 5$, contradicting $|E| \geq t_5^3(n)$. Hence, there is $I_t \sub J_t$ such that $|I_t| \geq (1 - \aA) n/5$ and for every $j \in I_t$ there exists some $q \in \{a,b,c\}$ such that $\{u, x_j^q, z\} \in E$ for at most $14$ vertices $z \in V$. Then the degree of $u$ is at most 
\[\tbinom{4}{2} (n/5 + \dD n)^2 + 4 \tbinom{n/5 + \dD n}{2} + 14n + 20 \aA n^2 - |I_t| (n-2)/2 < \dD_5^3(n) \,,\]
contrary to our assumption on the minimum degree of $\mc{F}$. We conclude that there are no bad vertices.
\end{proof}

Finally, we prove that all edges are good. Assume for the sake of contradiction that $\{u, v, w\} \in E$ is not a good edge. Without loss of generality $u, v \in A_1$ and $w \in A_1 \cup A_2$. Consider a maximal matching $\{x^3_j,x^4_j,x^5_j\}$, $j \in J$ in $E$ such that $x^q_j \in A_q$ for all $q \in \{3,4,5\}$, $j \in J$. Then $|J| \ge n/10$, similarly to before. Since $\mc{F}$ is $\mc{K}_{3,3}^3$-free, for every $j \in J$ there are $3 \leq q \leq 5$ and $y \in \{u, v, w\}$ such that $\{x^q_j,y\}$ is contained in at most $14$ edges. Without loss of generality there are at least $n/90$ indices $j \in J$ for which $\{x^3_j,u\}$ is contained in at most $14$ edges. Since $u$ is not bad, there are at most $40 \aA n^2$ bad or very bad edges incident with $u$. Hence, the degree of $u$ in $\mc{F}$ is at most $\binom{4}{2} (1/5 + \dD)^2 n^2 - n^2/91 + 40 \aA n^2 < \dD_5^3(n)$. This contradicts our assumption on the minimum degree in $\mc{F}$, so all edges are good. It follows that $\mc{F} \cong T_5^3(n)$. \qed

\section{Concluding remarks and open problems} \label{sec::openprob}

The natural open problem is to extend our results from $3$-graphs to general $r$-graphs. We would like to determine the Tur\'an number of the $r$-graph $\mc{K}_{r,r}^r$ with vertex set 
\[V(\mc{K}_{r,r}^r) = \{x_i, y_i : 1 \le i \le r \}\cup \{z_{ijk}: 1 \le i,j \le r, 1 \le k \le r-2\}\]
and edge set 
\[E(\mc{K}_{r,r}^r) = \{\{x_1, \dots, x_r\}, \{y_1, \dots, y_r\}\} \cup \{\{x_i, y_j, z_{ij1}, \dots, z_{ij(r-2)}\}: 1 \le i,j \le r\}.\]
The main difficulty seems to be in obtaining the analogue of Theorem \ref{th::main3int}, i.e.\ determining the maximum lagrangian of an intersecting $r$-graph. At first, one might think that $K^r_{2r-1}$ should be optimal, since this is the case when $r=3$. However, this has lagrangian $\binom{2r-1}{r} \brac{ \tfrac{1}{2r-1} }^r $, whereas stars (in which edges consist of all $r$-tuples containing some fixed vertex) give lagrangians that approach $\tfrac{1}{r!} \brac{ 1-\tfrac{1}{r} }^{r-1}$, which is better for $r \ge 4$. We conjecture that stars are optimal for $r \ge 4$ and that their blow-ups are extremal. Namely, we conjecture that the following hypergraph Tur\'an result holds. Let $S^r(n)$ be the $r$-graph on $n$ vertices with parts $A$ and $B$, where the edges consist of all $r$-tuples with $1$ vertex in $A$ and $r-1$ vertices in $B$, and the sizes of $A$ and $B$ are chosen to maximise the number of edges (so $|A| \sim n/r$). Write $s^r(n) = e(S^r(n))$.

\begin{conjecture}
$\ex(n, \mc{K}_{r,r}^r) = s^r(n)$ for $r \ge 4$ and sufficiently large $n > n_0(r)$. Moreover, if $n$ is sufficiently large and $G$ is a $\mc{K}_{r,r}^r$-free $r$-graph with $n$ vertices and $s^r(n)$ edges, then $G \cong S^r(n)$.
\end{conjecture}

More generally, our work suggests a direction of investigation in Extremal Combinatorics, namely to determine the maximum lagrangian for any specified property of $r$-graphs. For this paper, the property was that of being intersecting. This direction was already started by Frankl and F\"uredi~\cite{FF}, who considered the question of maximising the lagrangian of an $r$-graph with a specified number of edges. They conjectured that initial segments of the colexicographic order are extremal. Many cases of this have been proved by Talbot \cite{T}, but the full conjecture remains open.

\section*{Acknowledgement}

We would like to thank the anonymous referee for helpful comments.

\end{document}